\documentclass[12pt,reqno]{amsart}

\usepackage[margin=2cm]{geometry}

\usepackage{amssymb,amsfonts}
\usepackage[all,arc]{xy}
\usepackage{enumerate}
\usepackage{mathrsfs}
\usepackage{color}   
\usepackage{hyperref}
\hypersetup{
    colorlinks=true, 
    linktoc=all,     
    linkcolor=blue,  
    citecolor=red,
}
\usepackage{amsthm}
\usepackage{amsmath}
\usepackage{graphicx}
\usepackage{pgf,tikz}
\usetikzlibrary{positioning}

\newcommand{\bD}{\mathbf{D}}
\newcommand{\bx}{\boldsymbol{x}}
\newcommand{\by}{\boldsymbol{y}}

\newcommand{\sca}{\mathcal{A}}
\newcommand{\scd}{\mathcal{D}}
\newcommand{\scf}{\mathcal{F}}

\newcommand{\scw}{\mathcal{W}}

\newcommand{\scy}{\mathcal{Y}}
\newcommand{\scp}{\mathcal{P}}

\newcommand{\M}{\mathrm{M}}
\newcommand{\op}{\mathrm{op}}

\newcommand{\crc}{\mathscr{C}}
\newcommand{\crs}{\mathscr{S}}
\newcommand{\fd}{\mathfrak{D}}

\newcommand{\bc}{\mathbb C}

\newcommand{\ba}{\mathbb A}

\newcommand{\la}{\langle}
\newcommand{\ra}{\rangle}

\newcommand{\bs}{\backslash}

\newcommand{\lam}{\lambda}
\DeclareMathOperator{\Aut}{Aut}
\DeclareMathOperator{\Alg}{Alg}
\DeclareMathOperator{\res}{res}
\DeclareMathOperator{\End}{End}
\DeclareMathOperator{\ind}{ind}
\DeclareMathOperator{\GL}{GL}
\DeclareMathOperator{\Hom}{Hom}
\DeclareMathOperator{\Ind}{Ind}
\DeclareMathOperator{\SL}{SL}
\DeclareMathOperator{\Sh}{Sh}
\DeclareMathOperator{\St}{St}
\DeclareMathOperator{\Wh}{Wh}

\newcommand{\rres}{\mathrm{res}}

\newtheorem{Thm}{Theorem}[section]
\newtheorem{Prop}[Thm]{Proposition}
\newtheorem{Lem}[Thm]{Lemma}

\theoremstyle{definition}
\newtheorem{Def}[Thm]{Definition}

\theoremstyle{remark}
\newtheorem{Rem}[Thm]{Remark}

\title{Twisted doubling integrals for classical groups}
\author{Yuanqing Cai}
\address{Department of Mathematics, Kyoto University, Kitashirakawa Oiwake-cho, Sakyo-ku, Kyoto 606-8502, Japan}
\email{cai@math.kyoto-u.ac.jp}
\date\today
\thanks{This research was supported by the ERC, StG grant number 637912 and JSPS KAKENHI, grant number 19F19019.}
\subjclass[2010]{Primary 11F70; Secondary 11F55, 22E50, 22E55}
\keywords{Twisted doubling integrals, degenerate Whittaker coefficients, nilpotent orbits}

\begin{document}
\begin{abstract}
We describe the twisted doubling integrals of Cai-Friedberg-Ginzburg-Kaplan in a conceptual way. This also extends the construction to the quaternionic unitary groups. We carry out the unfolding argument uniformly in this article. To do so, we define a family of degenerate Whittaker coefficients that are suitable in this setup and study some of their properties. We also prove certain related global and local results that use the same tools.
\end{abstract}

\maketitle
\tableofcontents

\section{Introduction}

One of the fundamental problems in the theory of automorphic representations is to study analytic properties of automorphic $L$-functions. To this problem, Piatetski-Shapiro and Rallis \cite{PSR87} have discovered a family of zeta integrals which generalized the Godement-Jacquet zeta integral from $\GL(n)$ to an arbitrary simple classical group.  Their construction, known as the doubling method, includes the standard $L$-functions (twisted by Hecke characters) of all irreducible cuspidal automorphic representations of all simple classical groups.

The construction of Piatetski-Shapiro and Rallis was recently generalized in \cite{CFGK19}. This is a family of global integrals that represent the tensor product $L$-functions of a split classical group and a general linear group. (In what follows, this new family of integrals is referred as the \textit{twisted doubling integrals}.) The purpose of this article is to rephrase the twisted doubling integrals in a conceptual setting that is close to the description in \cite{PSR87}. This extends the construction to all classical groups, including the quaternionic unitary groups.

The main new ideas in \cite{CFGK19} are the use of the generalized Speh representations as the inducing data of a Siegel Eisenstein series and the use of a local and global model of degenerate type, which generalizes the Whittaker model. The global integral also uses a certain degenerate Whittaker coefficient (associated with nilpotent orbits) of a Siegel Eisenstein series. The Fourier coefficients that are used in the construction are indeed defined on the unipotent radicals of some parabolic subgroups. To give a conceptual description of the twisted doubling integrals, we first define a family of degenerate Whittaker coefficients that arise naturally in this setup in Sect. \ref{sec:degenerate Wh coeff}. This does not include all the possible degenerate Whittaker coefficients, but include sufficiently many of them to capture the nilpotent invariants of the representations involved.

The main body of this paper is to describe the twisted doubling integrals in a conceptual setup, and give a simplified and uniform version of the unfolding argument. Such an argument is already given in \cite{CFGK19} for symplectic groups. The basic ideas of unfolding are to break the sum in the Eisenstein series and rewrite the global integral as a sum over some double coset space. The technical part is to eliminate the contributions from almost all double cosets. The conceptual description does allow us to invoke some linear algebraic results (and avoid combinatorial argument) to give a neat presentation in some steps in the unfolding process. We refer the reader to Sect. \ref{sec:contribution in unfolding} for the tools we develop.

In this paper we also prove several relevant results. We calculate a certain Fourier coefficient of the Siegel Eisenstein series that are used in the global integrals. Our purpose is two-fold. First, such a calculation is necessary for the discussion of the normalization of intertwining operators, which is a key ingredient in the local theory of the twisted doubling integrals. Second, the calculation here is also simpler than the unfolding and serves as a test case for our method. We also prove analogous local results in the last section.

The twisted doubling integrals are expected to have far-reaching consequences. We would like to highlight a couple of follow-up directions of the present paper:
\begin{enumerate}
\item The Archimedean analogues of our local results are expected to be valid (at least for unitary representations) but the proofs require tools of different flavors. We plan to consider these results in a follow-up article.
\item The twisted doubling construction works in the case of covering groups (see \cite{Kaplan} in the case of covers of symplectic groups). The results and proofs presented in this paper can be extended to these cases with appropriate (non-trivial) changes. We plan to describe it in the framework of Brylinski-Deligne (BD) covering groups \cite{BD01}. This extension includes the case of BD covers of not necessarily split classical groups.

    The construction of Brylinski-Deligne only works for connected reductive groups but some of the classical groups considered in this paper are disconnected. It would also be interesting to see how the theory of Brylinski-Deligne can be extended to disconnected reductive groups.
\item One can further develop the local theory of the twisted doubling integrals to give a definition of $L$-factors and $\varepsilon$-factors of tensor type. In the original doubling integral setup, this is completed in \cite{LR05, Yamana14, Gan12, Kakuhama}. For twisted doubling integrals, this is done via $\gamma$-factors (following \cite{LR05}) for symplectic and special orthogonal groups in \cite{CFK}. (In \cite{CFK}, the twisted doubling integrals were also extended to the general spin groups). It is also desirable to define these factors as the greatest common divisors of the local integrals when the data vary. This will help locate the poles of the global tensor product $L$-functions. Globally, it will be interesting to see whether a new functoriality result for not necessarily quasi-split classical groups can be obtained using analytic properties of these $L$-functions and the Converse Theorem.
\item Some of the tools here can be used to calculate Fourier coefficients of general Eisenstein series on classical groups. The case we treat here is only a very particular instance. For the case treated in this paper, with some local analysis, one can prove certain cases of conjectures in \cite{Ginzburg06} regarding the nilpotent orbits attached to Eisenstein series. We hope to extend it to the general setup in a future article.
\end{enumerate}

The rest of this article is organized as follows. Sect. \ref{sec:groups} introduces the classical groups considered in this paper. In Sect. \ref{sec:degenerate Wh coeff} -- \ref{sec:general linear groups}, we define a special class of degenerate Whittaker coefficients and discuss some preliminary results that are used in the paper. We introduce the doubling variables in Sect. \ref{sec:doubling variables}. In Sect. \ref{sec:eisenstein series} we introduce the Siegel Eisenstein series and calculate a certain Fourier coefficient, as mentioned previously. The twisted doubling integrals are introduced in Sect. \ref{sec:twisted doubling} and the unfolding process is carried out in Sect. \ref{sec:unfolding}. Sect. \ref{sec:local results} proves some analogous local results.

\subsection*{Acknowledgment}
The author would like to thank Dmitry Gourevitch for very helpful conversations regarding degenerate Whittaker models and the theory of distributions. The author would also like to thank the referee for useful comments. Part of this work was carried out when the author was a postdoctoral fellow at the Weizmann Institute of Science. The author would like to thank the Institute for providing an excellent working environment.

\section{Preliminaries}

\subsection{Classical groups}\label{sec:groups}

We follow the setup as in \cite{Yamana14} Sect. 2.2.

The symbol $F$ will be used to denote a local field or global field of characteristic zero. If $F$ is a number field, then we denote by $\ba=\ba_F$ the adele ring of $F$, 
and by $\psi_F$ a nontrivial additive character character of $F\bs\ba$. If $F$ is local, then we
fix a nontrivial additive character $\psi_F$ of $F$.

By an involution of an algebra $D$ whose center $E$ contains $F$, we mean an arbitrary anti-automorphism $\rho$ of $D$ of order two under which $F$ is the fixed subfield of $E$. We denote the restriction of $\rho$ to $E$ also by $\rho$. We take a couple of $(D,\rho)$ belonging to the following five types:

\begin{enumerate}[(a)]
  \item $D=E=F$ and $\rho$ is the identity map;
  \item $D$ is a division quaternion algebra over $E=F$ and $\rho$ is the main involution of $D$;
  \item $D$ is a division algebra central over a quadratic extension $E$ of $F$ and $\rho$ generates $\mathrm{Gal}(E/F)$;
  \item $D=\M_2(E)$, $E=F$ and $\begin{pmatrix} a & b \\ c & d \end{pmatrix}^\rho = \begin{pmatrix} d & -b \\ -c & a \end{pmatrix}$;
  \item $D=\bD\oplus \bD^{\op}$, $E=F\oplus F$ and $(x,y)^\rho=(y,x)$, where $\bD$ is a division algebra central over $F$ and $\bD^{\op}$ is its opposite algebra.
\end{enumerate}

If $E=F$, we set $\psi=\psi_F$; if $E/F$ is an \'etale quadratic algebra, we set $\psi=\psi_F\circ \mathrm{tr}_{E/F}$. The global version is defined similarly. If $x$ is a square matrix with coordinates in $D$, then $\nu(x)\in E$ and $\tau(x)\in E$ stand for its reduced norm and reduced trace to the center $E$ of $D$.

We assume $D$ to be division if $F$ is a number field, so that $D$ is of type (d) (resp. (e)) will appear in our later discussion as a localization of a global $D$ of type (b) (resp. (c)).

Let $\epsilon$ be either $1$ or $-1$. We fix once and for all the triple $(D,\rho,\epsilon)$.

Let $W$ be a free left $D$-module of rank $n$. By an $\epsilon$-skew hermitian space we mean a structure $\scw=(W,\la\ ,\ \ra)$, where $\la \ , \ \ra$ is a $\epsilon$-skew hermitian form on $W$, that is, an $F$-bilinear map $\la \ , \ \ra: W\times W\to D$ such that
\[
\la x, y\ra^\rho = -\epsilon\la y , x \ra,\qquad \la ax,by\ra=a\la x,y\ra b^\rho,  \ (a,b\in D; \ x,y \in W).
\]
Such a form is called non-degenerate if $\la x ,W\ra=0$ implies that $x=0$. We assume that $\la \ , \ \ra$ is non-degenerate.

We denote the ring of all $D$-linear endomorphisms of $W$ by $\End_D(W)$ and set $\GL_D(W)=\End_D(W)^\times$. Note that $\GL_D(W)$ acts on $W$ on the right. We sometimes write $\GL_{W;D}$ for $\GL_D(W)$ for ease of notations. 
Let
\[
G=
\{
g\in \GL_D(W): \la xg,yg\ra = \la x,y\ra \text{ for all }x,y\in W
\}
\]
be the unitary group of $(W,\la \ , \ \ra)$, which is (possibly disconnected) reductive algebraic group defined over $F$. It is important to realize that $G$ always comes together with a space $W$ and a form $\la \ , \ \ra$. We usually just speak of $G$ and let the data $\scw=(W,\la\ ,\ \ra)$ be implicitly understood. When the dependence of $G$ on $\scw$ needs to be stressed, we write $G=G(\scw)$. If $F$ is local, then we shall deal with the representations of the group of $F$-rational points of $G$, while if $F$ is global, then we consider the localization and adelization of $G$.

We now give a list of notations that will be used repeatedly in the paper.
\begin{itemize}
  \item $[G]$: for an algebraic group $G$ over a number field $F$, we usually write $[G]$ for $G(F)\bs G(\ba)$ to save space.
  \item $G=G(F)$: for an algebraic group $G$ over a local field $F$, we usually write $G$ for $G(F)$ if there is no confusion.
  \item $g^{-1}(\pi)$: Let $H$ be a subgroup of $G$ and $(\pi,V)$ be a representation of $H$. Let $g\in G$. We define a representation $(g^{-1}(\pi),V)$ of the group $g^{-1}Hg$ whose action is given by $g^{-1}(\pi)(g^{-1}hg)v=\pi(h)v$ for $v\in V$.
  \item $\Ind$ and  $\ind$: induction and compact induction, both are normalized.
  \item $G^{ab}$: for a group $G$, $G^{ab}$ denotes its maximal abelian quotient.
  \item We use $1$ to denote the identity element in a group.
  \item $\dim W=\dim_D W$: rank of a free left $D$-module $W$.
\end{itemize}

\subsection{Some degenerate Whittaker coefficients}\label{sec:degenerate Wh coeff}

In this section we define a family of unipotent subgroups and characters. Let $R=F\bs \ba$ when $F$ is a number field and $R=F$ when $F$ is a local field. In either case, the continuous $F$-dual of $R$ is isomorphic to $F$.

\subsubsection{General linear groups}
We start with the case of $\GL_{W;D}$.
Let
\[
\scy:0\subset Y_1\subset Y_2\subset \cdots \subset Y_k\subset W
\]
be a flag of distinct subspaces of $W$. We sometimes write $Y_0=\{0\}$ and $Y_{k+1}=W$ for convenience. The stabilizer of $\scy$ is a parabolic subgroup $P(\scy)=M(\scy)\cdot N(\scy)$ with Levi component $M(\scy)$. Then as algebraic groups,
\[
N(\scy)^{ab}\cong \prod_{i=1}^k \Hom_D(Y_{i+1}/Y_i,Y_i/Y_{i-1}),\qquad u\mapsto (u_i)_{i=1}^{k}.
\]
To give a character of $N(\scy)(R)$, we specify an element in
\[
\sca=(A_1,\cdots, A_k)\in \prod_{i=1}^k \Hom_D(Y_{i}/Y_{i-1},Y_{i+1}/Y_i).
\]
More concretely, given such an $\sca$, we define a character $\psi_{\sca}$ of $N(\scy)(R)$ by
\[
\psi_{\sca}(u)=\psi\left(\sum_{i=1}^k \tau(u_i\circ A_i)\right).
\]

\subsubsection{Classical groups}\label{sec:FC of classical}
We now consider the case of $G(\scw)$ with a non-degenerate $\epsilon$-hermitian form.  We make the following assumptions:
\begin{itemize}
  \item $\dim W$ is even;
  \item $W$ admits a complete polarization $W=X\oplus X'$.
\end{itemize}
Let
\[
\scy:0\subset Y_1\subset Y_2\subset \cdots \subset Y_k
\]
be a flag of distinct totally isotropic subspaces of $W$. Let $N(\scy)$ be its unipotent radical. We now have two cases to consider:

Case 1: $Y_k$ is not maximally totally isotropic.

In this case, the flag $\scy$ can be extended to a long flag
\[
0\subset Y_1\subset \cdots \subset Y_k\subset Y_k^\perp \subset \cdots \subset Y_1^\perp \subset W.
\]
We put $Y_{k+i}=Y_{k-i+1}^\perp$ for $i=1,\cdots, k+1$ in order to have a uniform description occasionally. Then as algebraic groups
\[
N(\scy)^{ab}\simeq \prod_{i=1}^k \Hom_D(Y_{i+1}/Y_i,Y_i/Y_{i-1}).
\]
In this case, we use
\[
\sca=(A_1,\cdots,A_k)\in \prod_{i=1}^k \Hom_D(Y_{i}/Y_{i-1},Y_{i+1}/Y_i)
\]
to give a character of $N(\scy)(R)$ as in the general linear case. Note that we can also define $A_i\in \Hom_D(Y_{i+1}/Y_i,Y_i/Y_{i-1})$ for $i=k+1,\cdots, 2k$ as minus the $\rho$-dual of  $A_{2k+1-i}$. The same comment applies to $u_i$ as well.

Case 2: $Y_k$ is maximally totally isotropic. This case is more complicated. It is used in Sect. \ref{sec:eisenstein series}. In this case, the flag can be extended to a long flag
\[
0\subset Y_1\subset \cdots \subset Y_k\subset Y_{k-1}^\perp \subset \cdots \subset Y_1^\perp \subset W.
\]
We put $Y_{k+i}=Y_{k-i}^\perp$ for $i=1,\cdots, k$.
The $\epsilon$-hermitian form on $W$ induces a $\epsilon$-hermitian form on $Y_{k-1}^\perp/Y_{k-1}$. The space $Y_k/Y_{k-1}$ is a maximal totally isotropic subspace of $Y_{k-1}^\perp/Y_{k-1}$. We assume that $N(Y_{k+1}/Y_k) \subset G(Y_{k-1}^\perp/Y_{k-1})$ is nonzero. Then the map $u\mapsto u-1$ induces an injection $N(Y_{k+1}/Y_k)\to \Hom_D(Y_{k+1}/Y_k,Y_k/Y_{k-1})$. We denote the image as $\Hom_D^\ast(Y_{k+1}/Y_k,Y_k/Y_{k-1})$.

We have
\[
N(\scy)^{ab}\simeq \left( \prod_{i=1}^{k-1} \Hom_D(Y_{i+1}/Y_i,Y_i/Y_{i-1})\right)\times \Hom_D^\ast(Y_{k+1}/Y_k,Y_k/Y_{k-1}) .
\]
We choose an element
\[
\sca=(A_1,\cdots,A_{k-1},A_k)\in \left( \prod_{i=1}^{k-1} \Hom_D(Y_{i}/Y_{i-1},Y_{i+1}/Y_i)\right) \times \Hom_D(Y_k/Y_{k-1},Y_{k+1}/Y_k)
\]
to define a character of $N(\scy)(R)$.
Note that different $A_k$ might give the same character. 
We can define $A_i,u_i$ for $i> k$ as in the previous case.

\begin{Rem}
It is possible that $\Hom_D^\ast(Y_{k+1}/Y_k,Y_k/Y_{k-1})=\{0\}$. In this case, $N(\scy)^{ab}$ admits a different description. We will not encounter this case in this paper. Such a case is already treated in \cite{LR05} Sect. 6.
\end{Rem}

\subsubsection{Fourier coefficients}
Assume now we have a pair $(N(\scy),\psi_\sca)$. Globally, for an irreducible automorphic representation $\pi$ of $G(\ba)$ or $\GL_{W;D}(\ba)$, we define the $(N(\scy),\psi_\sca)$-Fourier coefficient of $\phi\in \pi$ as
\[
\phi^{N(\scy),\psi_\sca}(g)=\int\limits_{[N(\scy)]} \phi(ug)\psi_{\sca}(u) \ du.
\]
Locally, we consider the space $\Hom_{N(\scy)(F)}(\pi,\psi_{\sca})$ of $(N(\scy),\psi_{\sca})$-functional for an admissible representation $\pi$ of $G(F)$ or $\GL_{W;D}(F)$. In the local non-Archimedean case, we also consider the twisted Jacquet module $J_{N(\scy)(F),\psi_{\sca}}(\pi)$.

\begin{Rem}
The coefficients defined here are degenerate Whittaker coefficients in the sense of \cite{MW87,GGS17}, for a suitable choice of Whittaker pair.
\end{Rem}

\subsection{Contributions in the unfolding argument}\label{sec:contribution in unfolding}
The result in this section is motivated by the following question. We describe the question globally. Let $G=G(\scw)$ be a classical group satisfying the assumptions in Sect. \ref{sec:FC of classical}. Let $\pi$ be an irreducible automorphic representation of $\GL_{X;D}(\ba)$. Consider the normalized induced representation $\Ind_{P(X)(\ba)}^{G(\ba)}(\pi\cdot \nu^s)$. Let $\tilde\phi^{(s)}$ be a section of $\Ind_{P(X)(\ba)}^{G(\ba)}(\pi\cdot \nu^s)$ and $\phi^{(s)}(g)=\tilde\phi^{(s)}(g;1)$ denote the value at the identity. One can define an Eisenstein series $E(\phi^{(s)})$ via
\[
E(\phi^{(s)})(g)=\sum_{\gamma \in P(X)(F)\bs G(F)}\phi^{(s)}(\gamma g)
\]
when $\Re s\gg 0$. We would like to calculate the Fourier coefficient
\[
E(\phi^{(s)})^{N(\scy),\psi_\sca}(g)=\int\limits_{[N(\scy)]} E(\phi^{(s)})(ug) \psi_{\sca}(u) \ du
\]
when $\Re s\gg0$.
By a standard unfolding argument, this becomes a sum over the double coset space $P(X)(F)\bs G(F)/N(\scy)(F)$. The quotient $P(X)(F)\bs G(F)$ can be identified with the variety $\Omega(W)$ of maximal totally isotropic subspaces of $W$. Thus each contribution can be indexed by an element in $\Omega(W)$. For instance, if $L=X\gamma$, then the contribution corresponds to this orbit contains
\[
\int\limits_{[P(L)\cap N(\scy)]} \phi^{(s)}(ug) \psi_{\sca}(u) \ du
\]
as an inner integral. The results presented in this section can be used to show that the contributions from certain orbits are zero.

\begin{Prop}\label{prop:first vanishing criterion}
Let $L$ be a maximal totally isotropic subspace of $W$. The following are equivalent:
\begin{enumerate}[\normalfont(1)]
  \item $\psi_{\sca}|_{(N(L)\cap N(\scy))(R)}$ is trivial;
  \item  $\tau(u_i\circ A_i)= 0$ for all $u\in (N(L)\cap N(\scy))(R)$ and $i=1,\cdots, 2k$ or $2k+1$;
  \item $A_i(L\cap Y_i/L\cap Y_{i-1})\subset L\cap Y_{i+1}/L\cap Y_{i}$ for all $i=1,\cdots, 2k$ or $2k+1$.
\end{enumerate}
\end{Prop}

\begin{proof}
The equivalent between (1) and (2) is clear. We now show that (2) and (3) are equivalent.

Let $u\in N(L)\cap N(\scy)$. Recall that $u_i:Y_{i+1}/Y_i\to Y_i/Y_{i-1}$ is obtained from
\[
Y_{i+1} \xrightarrow{u} Y_i\to Y_i/Y_{i-1}.
\]
Both $Y_i$ and $L\cap Y_{i+1}$ are in the kernel of this map. This implies that the kernel of $u_i$ contains the image of $L\cap Y_{i+1}$ in $Y_{i+1}/Y_i$, which is
\[
((L\cap Y_{i+1})+Y_i)/Y_i\cong L\cap Y_{i+1}/L\cap Y_i.
\]
On the other hand, the image of $u|_{Y_{i+1}}$ is in $L\cap Y_i$.

We now conclude that, for a map $u_i\in \Hom_D(Y_{i+1}/Y_i,Y_i/Y_{i-1})$, its kernel contains $L\cap Y_{i+1}/L\cap Y_i$ and its image of $u_i$ is in $(L\cap Y_i+Y_{i-1})/Y_{i-1}\cong L\cap Y_{i}/L\cap Y_{i-1}$.

We now assume that (3) holds. Then
\[
u_i\circ A_i(L\cap Y_i/L\cap Y_{i-1})\subset u_i(L\cap Y_{i+1}/L\cap Y_{i})=0
\]
\[
u_i\circ A_i(Y_i/Y_{i-1})\subset u_i(Y_{i+1}/Y_{i})\subset L\cap Y_{i+1}/L\cap Y_{i}.
\]
Thus $\tau(u_i\circ A_i)=0$.

We now assume that (2) holds.  Then the image of $A_i$ under the canonical projection
\[
\Hom_D(Y_{i}/Y_{i-1},Y_{i+1}/Y_i)\to \Hom_D(L\cap Y_i/L\cap Y_{i-1},(Y_{i+1}/Y_i)/(L\cap Y_{i+1}/L\cap Y_i))
\]
is zero.
This implies that $A_i(L\cap Y_i/L\cap Y_{i-1})\subset   L\cap Y_{i+1}/L\cap Y_i$.
\end{proof}

Note that $N(L)\cap N(\scy)\bs P(L)\cap N(\scy)$ can be viewed as a subgroup of $\GL_D(L)$ and it is the unipotent radical of the parabolic subgroup of $\GL_D(L)$ stabilizing
\[
0\subset L\cap Y_1 \subset L\cap Y_2 \subset \cdots \subset L\cap Y_1^\perp\subset L.
\]
Assume that $\psi_{\sca}|_{(N(L)\cap N(\scy))(R)}$ is trivial. On the one hand, by Proposition \ref{prop:first vanishing criterion}, for $i\geq 1$, $A_i$ gives an element in $\Hom_D(L\cap Y_{i}/L\cap Y_{i-1},L\cap Y_{i+1}/L\cap Y_{i})$, which is again denoted by $A_i$.
On the other hand, the character $\psi_{\sca}$ can be viewed as a character on $(N(L)\cap N(\scy)\bs P(L)\cap N(\scy))(R)$. This character is given by the restriction of $A_i$ to $\Hom_D(L\cap Y_{i}/L\cap Y_{i-1},L\cap Y_{i+1}/L\cap Y_{i})$ for $i\geq 1$.

We now record a result regarding the structure of $\Omega(W)$ under the action of either parabolic subgroups or their unipotent radicals. The proof of the following lemma is left to the reader.

\begin{Lem}\label{lem:orbit under para and nil}
Let $L, L' \in \Omega(W)$.
\begin{enumerate}[\normalfont(1)]
  \item  There exists $p\in P(\scy)$ such that $L=L'p$ if and only if the intersections $L\cap Y_i$ and $L'\cap Y_i$ has the same dimension for $i=1,\cdots,k$.
  \item There exists $u\in N(\scy)$ such that $L=L'u$ if and only if $L\cap Y_i/L\cap Y_{i-1}=L'\cap Y_i/L'\cap Y_{i-1}$ as subsets of $Y_i/Y_{i-1}$ for $i=1,\cdots,k+1$.
\end{enumerate}
\end{Lem}

\subsection{The case of general linear groups}\label{sec:general linear groups}
\subsubsection{Representations of type $(k,n)_D$}

In this section, we consider the group $\GL_{kn,D}$. We usually identify it with $\GL_{W;D}$ where $\dim W=kn$ for convenience. The purpose of this section is to introduce the notion of representations of type $(k,n)_D$, both locally and globally. These representations are supported on a suitable nilpotent orbit and admits unique models of degenerate type. When $D$ is a field, the generalized Speh representations are examples of such representations.

Let $\scy$ be a flag of distinct subspaces of $W$.

\begin{Def}\label{def:orbit kn}
We say a pair $(N(\scy),\psi_{\sca})$ lies in the orbit $(k^n)_D$ if $\scy$ is of the form
\[
0\subset Y_1 \subset \cdots \subset Y_{k-1}\subset W.
\]
and for $i=1,\cdots,k-1$, $\dim Y_i=ni$ and $A_i$ is an isomorphism.
\end{Def}

The stabilizer of a pair that lies the orbit $(k^n)_D$ is isomorphic to $\GL_{n,D}$.

\begin{Def}
We say a pair  $(N(\scy),\psi_{\sca})$ lies in an orbit higher than $(k^n)_D$ if
\[
A_{i+k-1}\circ \cdots \circ A_{i}\neq 0
\]
for some $i$.
\end{Def}

Note that this implies that there are at least $k$ terms in the flag $\scy$.

\begin{Rem}
The nilpotent orbits of $\GL_{n,D}$ can be classified by partitions $n$. Note that $\sca$ for the group $\GL_{W;D}$ can be lifted to a nilpotent element $\tilde\sca\in \End_D(D^{kn})$. The above two definitions are compatible with this classification.
\end{Rem}

\begin{Def}\label{def:local (k,n)}
We say an irreducible admissible representation $\theta$ of a local group $\GL_{W;D}$ is of type $(k,n)_D$ if the following two conditions hold:
  \begin{enumerate}
    \item For a pair $(N(\scy),\psi_{\sca})$ that lies in the orbit $(k^n)_D$,
        \[
        \dim\Hom_{N(\scy)}(\theta,\psi_{\sca})=1.
        \]
    \item For any pair $(N(\scy),\psi_{\sca})$ that lies in an orbit higher than $(k^n)_D$,
        \[
        \dim \Hom_{N(\scy)}(\theta,\psi_{\sca})=0.
        \]
  \end{enumerate}
\end{Def}

\begin{Rem}
In this paper, we need the multiplicity one condition in Definition \ref{def:local (k,n)} (1) when deducing that the global zeta integral is Eulerian.

If (1) in Definition \ref{def:local (k,n)} is replaced by the weaker condition $\Hom_{N(\scy)}(\theta,\psi_{\sca})\neq 0$, we say that the nilpotent orbit attached to $\theta$ is $(k^n)_D$.

Assume that $D$ is a field. The above two notions are related by the following result. Let $\theta$ be an irreducible admissible representation of $\GL_{W;D}$ whose nilpotent orbit is $(k^n)_D$. Assume further that $\theta$ is unitarizable when $D$ is Archimedean. Then $\theta$ is of type $(k,n)_D$, i.e. multiplicity one holds (see \cite{MW87}, \cite{Zelevinsky80} Corollary 8.3 and \cite{GGS17} Corollary G). (However, this result is not true when $D$ is not a field.)
\end{Rem}

\begin{Rem}
By Frobenius reciprocity, $\Hom_{N(\scy)}(\theta,\psi_{\sca})\simeq \Hom_{\GL_{W;D}}(\theta,\Ind_{N(\scy)}^{\GL_{W;D}}(\psi_{\sca}))$. An element in the latter space is called a $(N(\scy),\psi_{\sca})$-model for $\theta$. For a representation $\theta$ of type $(k,n)_D$, we write $\Wh_{N(\scy),\psi_{\sca}}(\theta)$ for the image of a nonzero map in $\Hom_{\GL_{W;D}}(\theta,\Ind_{N(\scy)}^{\GL_{W;D}}(\psi_{\sca}))$.
\end{Rem}

\begin{Def}\label{def:global (k,n)}
  We say an irreducible automorphic representation $\theta$ of $\GL_{W;D}(\ba)$ is of type $(k,n)_D$ if the following conditions hold:
  \begin{enumerate}
    \item The representation supports a nonzero $(N(\scy),\psi_{\sca})$-Fourier coefficient such that the pair lies in the orbit $(k^n)_D$.
    \item For any pair $(N(\scy),\psi_{\sca})$ that lies in an orbit higher than $(k^n)_D$, the $(N(\scy),\psi_A)$-Fourier coefficient vanishes identically.
    \item The local component $\theta_v$ is a representation of type $(k,n)_D$ for every place $v$.
\end{enumerate}
We also say that the nilpotent orbit attached to $\theta$ is $(k^n)_D$ if only parts (1) and (2) hold.
\end{Def}

\begin{Rem}
When $D$ is a field, (3) in Definition \ref{def:global (k,n)} seems to be redundant (but we cannot find a reference for this). If $D$ is not a field, (3) does not follow from the other two conditions.
\end{Rem}

\begin{Rem}
By \cite{Ginzburg06} Proposition 5.3, \cite{JL13} Theorem 1 and \cite{CFK} Theorem 5, the generalized Speh representations are representations of type $(k,n)_D$.
\end{Rem}

\begin{Rem}
There are other definitions of nilpotent orbits attached to a representation in the literature. We refer the reader to \cite{GS19} for a comprehensive account of history. Our definition here only uses a small subclass of all possible coefficients. However, this definition is equivalent to the usual definition (see \cite{Ginzburg06} Definition 2.1 for example). We explain the reason briefly here. Notice that the coefficients given in this paper give at least one coefficient for each nilpotent orbit.

The coefficient we use for the orbit $(k^n)_D$ is a neutral coefficient. By \cite{GGS17} Theorem A, this implies the coefficient is nonvanishing for every Whittaker pair that lies in the orbit $(k^n)_D$. We now show that our vanishing condition implies the vanishing condition for every Whittaker pair in \cite{GGS17} that lies in an orbit higher than $(k^n)_D$. If not, then for some orbit higher than $(k^n)_D$, there is a Whittaker pair whose corresponding degenerate Whittaker coefficient is nonvanishing for $\theta$. We choose a maximal one $\mathcal{O}$ among such orbits. By \cite{GGS} Theorem 8.2.1, the Fourier coefficient is nonvanishing for every Whittaker pair that lies in the orbit $\mathcal{O}$. This contradicts with our assumption.
\end{Rem}

\begin{Lem}\label{lem:eulerian}
Let $\theta=\otimes'_v\theta_v$ be an irreducible automorphic representation of $\GL_{W;D}(\ba)$ which is of type $(k,n)_D$, and consider a $(N(\scy),\psi_{\sca})$-Fourier coefficient for a pair that lies in the orbit $(k^n)_D$. Then for a decomposable $\phi=\otimes_v \phi_v$, there exists $f_v\in \Wh_{N(\scy),\psi_{\sca}}(\theta_v)\subset \Ind_{N(\scy)(F_v)}^{\GL_{W;D}(F_v)}(\psi_{\sca})$ such that
\[
\phi^{N(\scy),\psi_{\sca}}(g)=\prod_v f_v(g_v),
\]
where $g=(g_v)_v\in \GL_{W;D}(\ba)$.
\end{Lem}

\begin{proof}
The Fourier coefficient
\[
\pi\to \bc,\qquad \phi\mapsto \phi^{N(\scy),\psi_{\sca}}(1)
\]
defines a functional $\lam\in \Hom_{N(\scy)(F\bs \ba)}(\theta,\psi_{\sca})$. From Definition \ref{def:local (k,n)}, one can show that for $\phi=\otimes_v \phi_v$, there exists $\lam_v\in \Hom_{N(\scy)(F_v)}(\theta_v,\psi_{\sca})$ such that
\[
\lam(\phi)=\prod_v \lam_v(\phi_v).
\]
Thus for any $g\in \GL_{W;D}(\ba)$,
\[
\phi^{N(\scy),\psi_{\sca}}(g)=\lam(\theta(g) \phi)=\prod_v \lam_v(\theta_v(g_v) \phi_v).
\]
We take $f_v(g_v)=\lam_v(\theta_v(g_v)\phi_v)$ and it is easy to check that $f_v\in \Ind_{N(\scy)(F_v)}^{\GL_{W;D}(F_v)}(\psi_{\sca})$.
\end{proof}

\subsubsection{Invariance under stabilizer}\label{sec:invariance under stab}

We continue assuming that $\dim W=kn$ and the representation $\theta$ of $\GL_{W;D}$ is of type $(k,n)_D$. We consider the following flag $\scy$:
\[
0\subset Y_1 \subset \cdots \subset Y_{k-1}\subset W
\]
such that $\dim Y_i=ni$, $A_2,\cdots, A_{k-1}$ are isomorphisms, the rank of $A_{1}$ is $a>0$ (which might not be of full rank).

We now show that Fourier coefficients defined by such a pair enjoy an extra invariance property. We start with the case of $a=n$. Recall that the stabilizer $\St_{\sca}$ of a pair $(N(\scy),\psi_{\sca})$ that lies in the orbit $(k^n)_D$ is isomorphic to $\GL_{n,D}$. We start with the local version.

\begin{Lem}\label{lem:invariance I local}
Let $\theta$ be an irreducible admissible representation of $\GL_{W;D}(F)$ that is of type $(k,n)_D$.
\begin{enumerate}
\item The stabilizer $\St_{\sca}$ acts on $\Hom_{N(\scy)}(\theta,\psi_{\sca})$ via a character $\chi_\theta: F^\times\to \bc^\times$.
\item For $f\in \Wh_{N(\scy),\psi_{\sca}}(\theta)$,
\[
f(gh)=\chi_{\theta}(\nu(g))f(h)
\]
for $g\in \St_{\sca}(F)$ and $h\in \GL_{W;D}(F)$.
\end{enumerate}
\end{Lem}

\begin{proof}
Indeed, $\Hom_{N(\scy)}(\theta,\psi_{\sca})$ is a one-dimensional representation of $\St_{\sca}(F)\simeq \GL_{n,D}(F)$. It must be trivial on $\SL_{n,D}(F)$ and therefore factors through $\nu$. Thus, this representation is a character $\chi_\theta$. In other words, $\lam(\theta(g)\phi)=\chi_\theta(\nu(g))\lam(\phi)$ for any $\phi\in \theta$. This proves the first part. The second part follows immediately.
\end{proof}

Here is the global version.

\begin{Lem}\label{lem:invariance I global}
Let $\theta$ be an irreducible unitary automorphic representation of $\GL_{W;D}(\ba)$ that is of type $(k,n)_D$. Then there is a character $\chi_\theta: F^\times \bs \ba^\times \to \bc^\times$ such that, for any $\phi\in \theta$,
\[
\phi^{N(\scy),\psi_{\sca}}(gh)=\chi_{\theta}(g)\phi^{N(\scy),\psi_{\sca}}(h)
\]
for any $g\in \St_{\sca}(F\bs \ba)$ and $h\in \GL_{W;D}(\ba)$.
\end{Lem}

\begin{proof}
Fix $h\in \GL_{W;D}(\ba)$. By Lemma \ref{lem:eulerian} and Lemma \ref{lem:invariance I local}, there exists $\chi_{\theta_v}:F_v^\times \to \bc^\times$ such that
\[
\phi^{N(\scy),\psi_{\sca}}(gh)=\prod f_v(g_v h_v)=\prod_v \chi_{\theta_v}(\nu(g_v))f_v(h_v) \text{ for any }g_v\in \St_{\sca}(F_v).
\]
Thus for any $g\in \St_{\sca}(\ba)$, $\phi^{N(\scy),\psi_{\sca}}(gh) =\chi_\theta(\nu(g))\phi^{N(\scy),\psi_{\sca}}(h)$ for a character $\chi_{\theta}$ of $\ba^\times$. It is easy to check that $\chi_\theta$ is trivial on $F^\times$.
\end{proof}

We now consider the case $a<n$. The results here are weaker. In particular, locally we only consider non-Archimedean cases.

Define $S_{\sca}$ to be the subgroup of $M(\scy)$:
\[
N(\mathrm{Ker}(A_1))\times \{1\} \times\cdots \times \{1\} \subset \GL(Y_1)\times \GL(Y_2/Y_1) \times \cdots \times \GL(W/Y_{k-1}).
\]

\begin{Lem}
The group $S_{\sca}$ is in the stabilizer of the pair $(N(\scy),\psi_{\sca})$.
\end{Lem}

\begin{proof}
This is straightforward.
\end{proof}



We now show that $\phi^{N(\scy),\psi_{\sca}}(g)$ is left-invariant under $[S_{\sca}]$. We will explain the necessary modification in the non-Archimedean case. 
Note that $S_{\sca}$ is not the full stabilizer.

\begin{Prop}\label{prop:invariance}
For $\phi\in \theta$,
\[
\phi^{N(\scy),\psi_{\sca}}(gh)=\phi^{N(\scy),\psi_{\sca}}(h)
\]
for any $g\in S_{\sca}(F\bs \ba)$ and $h\in \GL_{W;D}(\ba)$.
\end{Prop}

\begin{proof}
Note that $N(\mathrm{Ker}(A_1))(F)$ is isomorphic to $[N(\mathrm{Ker}(A_1))]^\vee$ once we choose an additive character $\psi$. We denote the bijection by $\gamma \mapsto \psi_\gamma$.
We can perform Fourier expansion of $\phi^{N(\scy),\psi_{\sca}}(g)$ along the abelian group $N(\mathrm{Ker}(A_1))$ to obtain
\[
\phi^{N(\scy),\psi_{\sca}}(g)=\sum_{\gamma \in N(\mathrm{Ker}(A_1))(F)} \
\int\limits_{[N(\mathrm{Ker}(A_1))]} \phi^{N(\scy),\psi_{\sca}}(ug) \psi_\gamma( u) \ du.
\]
Each term is an $(N(\scy^\dagger),\psi_{\sca^\dagger})$-Fourier coefficient, where the flag $\scy^\dagger$ is
\[
0\subset \mathrm{Ker}(A_1) \subset Y_1 \subset Y_2 \subset \cdots \subset Y_{k-1},
\]
and $\sca^\dagger$ is
\[
(\gamma, \bar A_1,\cdots, A_{k-1})\in \Hom_D(\mathrm{Ker}(A_1),Y_1/\mathrm{Ker}(A_1)) \times \Hom_D(Y_1/\mathrm{Ker}(A_1),Y_2/Y_1) \times \cdots
\]
Here $\bar A_1$ is the map $\bar A_1:Y_1/\mathrm{Ker}(A_1)\cong \mathrm{Im}(A_1)$.
The pair $(N(\scy^\dagger),\psi_{\sca^\dagger})$ lies in an orbit higher than $(k^n)_D$ as long as  $\gamma\in \Hom_D(\mathrm{Ker}(A_1),Y_1/\mathrm{Ker}(A_1))$ is nonzero. As $\theta$ is a representation of type $(k,n)_D$, this shows that
\[
\phi^{N(\scy),\psi_{\sca}}(g)=
\int\limits_{[N(\mathrm{Ker}(A_1))]} \phi^{N(\scy),\psi_{\sca}}(ug)  \ du
\]
and the result follows.
\end{proof}

We now state the non-Archimedean version.

\begin{Lem}\label{lem:invariance II local}
As a representation of $S_{\sca}$, the twisted Jacquet module $J_{N(\scy),\psi_{\sca}}(\theta)$ is trivial.
\end{Lem}

\begin{proof}
The argument above works, with the help of \cite{BZ76} Lemma 5.10 (see also \cite{GGS} Lemma 4.1.1).
\end{proof}

\begin{Rem}
The Lemma of Bernstein-Zelevinsky fails in the Archimedean case. The best we can get is \cite{GGS} Proposition 3.0.1. However, Lemma \ref{lem:invariance II local} is expected to be true when $\theta$ is unitary.
\end{Rem}

\begin{Rem}
The Fourier coefficient $\phi^{N(\scy),\psi_{\sca}}(g)$ is invariant under a larger subgroup of the stabilizer. What we prove here is sufficient for applications in this paper.
\end{Rem}

\section{Doubling variables}\label{sec:doubling variables}
Let $\scw=(W,\la \ , \ \ra)$ be one of the $\epsilon$-skew hermitian forms described in Sect. \ref{sec:groups}. Let $k$ be a fixed positive integer. Put $W^{\square,k}=W^{\oplus 2k}$. We usually write
\[
W^{\square,k} = W_{1,+}\oplus W_{2,+}\oplus \cdots \oplus  W_{k,+}\oplus W_{k,-}\oplus \cdots \oplus W_{2,-} \oplus W_{1,-}
\]
to distinguish the copies of $W$ in $W^{\square,k}$. We write an element in $W^{\square,k}$ as
\[
(\bx;\by)=(x_1,\cdots, x_k; y_{k},\cdots, y_{1}),\qquad x_i\in W_{i,+}, \ y_i\in W_{i,-}.
\]
Define an $\epsilon$-skew hermitian form $\la \ , \ \ra^{\square,k} $ on $W^{\square,k}$ by
\[
\la (\bx;\by), (\bx';\by') \ra^{\square,k} = \sum_{i=1}^k(\la x_i,x'_i\ra-\la y_{i},y'_{i}\ra) \qquad (x_i,x'_i\in W_{i,+}; y_{i},y'_{i}\in W_{i,-}).
\]
Let $G^{\square,k}$ denote the unitary group of $(W^{\square,k}, \la \ , \ \ra^{\square,k})$.

For $W^\square =W_+\oplus W_-$, let
\[
W^\nabla=\{(x,-x)\in W_{+}\oplus W_{-}: x\in W\}
\]
be the graph of minus the identity map from $W$ to $W$, and
\[
W^{\Delta}=\{(x,x)\in W_{+}\oplus W_{-}: x\in W\}.
\]
be the graph of the identity map. Given $x\in W$, we write
\[
x^\Delta=(x ,x)\in W^\Delta \text{ and } x^\nabla=(x,-x)\in W^\nabla.
\]
We have the following observations:
\begin{enumerate}
\item For each $i$, $W_i^\square=W_{i,+}\oplus W_{i,-}=W^\Delta_{i}+ W_i^\nabla.$
Both $W_i^\Delta$ and $W_i^\nabla$ is totally isotropic in $W^{\square,k}$.
\item The space $W^\Delta$ is  isomorphic to $W$ as vector spaces via
\[
W^\Delta \simeq W, \qquad (x,x)\mapsto x.
\]
The space $W^\nabla$ is identified with $W$ via $(x,-x)\mapsto 2x$. Thus, we can view $G(\scw)$ as a subgroup of $\GL_D(W^\Delta)$ or $\GL_D(W^\nabla)$, and identify $\Hom_D(W_i^\nabla,W_j^\nabla)$ with $\End_D(W)$.

\end{enumerate}

Define
\[
W^{\Delta,k}=W_1^\Delta \oplus \cdots \oplus W_k^\Delta,\qquad W^{\nabla,k}=W_1^\nabla \oplus \cdots \oplus W_k^\nabla.
\]
Both spaces are totally isotropic in $W^{\square,k}$ and $W^{\square,k} =W^{\Delta,k}+W^{\nabla,k}$. This is a complete polarization of $W^{\square,k}$.

\section{Eisenstein series}\label{sec:eisenstein series}
Let $D$ be a division algebra over a number field $E$ of the first three types referred to in Sect. \ref{sec:groups}.
We now describe the Eisenstein series that appear in the global construction.

Let $\theta$ be an irreducible unitary automorphic representation of $\GL_{kn,D}(\ba)$ that is of type $(k,n)_D$. Unless otherwise specified, we let $P=P(W^{\Delta,k})$. We define the normalized global induced representation $I(s,\theta)=\Ind_{ P(\ba)}^{ G^{\square,k}(\ba)}(\theta\cdot \nu^s)$.

For any holomorphic section $\tilde\phi^{(s)}$ of $I(s,\theta)$, we write $\phi^{(s)}(g)=\tilde\phi^{(s)}(g;1)$ to be the value at the identity. We form the associated Eisenstein series $E(\phi^{(s)})$ on $G^{\square,k}(F)\bs G^{\square,k}(\ba)$ by
\[
E(\phi^{(s)})(g)=\sum_{\gamma\in P(F)\bs G^{\square,k}(F)} \phi^{(s)}(\gamma g).
\]
The Eisenstein series converges for $\Re s\gg 0$. By the theory of Eisenstein series, it can be continued to a meromorphic function in $s$ on all of $\bc$ satisfying a functional equation.

\subsection{Fourier coefficients of certain representations}\label{sec:FC of Eisenstein series}

In this section we consider a certain Fourier coefficient of the Eisenstein series $E(\phi^{(s)})$. The purpose here is two-fold. First, we give this example as a test case for the tools we develop in the Sect. \ref{sec:contribution in unfolding}. The unfolding argument is more involved. Second, this is also a necessary ingredient in the normalization of the intertwining operators. This would appear in the local theory of the twisted doubling integrals.

Recall that the nilpotent orbit attached to $\theta$ on $\GL_{kn,D}(\ba)$ is $(k^n)_D$. The induced orbit of $(k^n)_D$ to $G^{\square,k}$ is $((2k)^n)_D$ when identified with partitions, except in the case of odd orthogonal group, where we need to take its `collapse'. \textbf{In this section we exclude the case of odd orthogonal groups, just as in \cite{LR05}}. The necessary modification will be explained in a forthcoming article.

We now describe the coefficient more concretely. Let $P_{\scw,k}^{\heartsuit}=M_{\scw,k}^{\heartsuit}\cdot N_{\scw,k}^{\heartsuit}$ be the parabolic subgroup of $G^{\square,k}$ stabilizing
\[
0\subset W_k^\nabla \subset W_{k-1}^\nabla \oplus W_k^\nabla \subset \cdots \subset W^{\nabla,k}.
\]
We usually write $P_{\scw}^{\heartsuit}=M_{\scw}^{\heartsuit}\cdot N_{\scw}^{\heartsuit}$ as the dependence on $k$ does not appear in this section. We reindex the flag as
\[
0\subset Y_1\subset  Y_2 \subset \cdots \subset Y_k.
\]
This extends to
\[
0\subset Y_1\subset  Y_2 \subset \cdots \subset Y_k \subset Y_{k-1}^\perp \subset \cdots \subset Y_1^\perp\subset W^{\square,k}.
\]
We adopt the convention in the Sect. \ref{sec:FC of classical} Case 2 for terms $Y_i,A_i$ for $i>k$. To define a character on $[N_{\scw}^{\heartsuit}]$, we have to choose
\[
A_i:Y_i/Y_{i-1}\to Y_{i+1}/Y_i,\qquad i=1,\cdots, k.
\]
For $i=1,\cdots,k-1$, we can choose $A_i$ `canonically' as the identity map when identifying $Y_i/Y_{i-1}$ with $W$ via the obvious projection. The only real choice is an isomorphism $A_{k}\in\Hom_D(W^{\nabla},W^{\Delta})$. (Note that this is always possible for the cases considered in this section.) In any case, what is important is that $A_i,i=1,\cdots, 2k-1$ is an isomorphism.

\begin{Thm}\label{thm:unfolding of fourier coefficient}
When $\Re s\gg 0$,
\[
\int\limits_{[N_{\mathcal{W}}^{\heartsuit}]} E(\phi^{(s)})(ug) \psi_\sca(u) \ du=\int\limits_{N_{\scw}^{\heartsuit}(F) \cap P(F)\backslash N_{\scw}^{\heartsuit}(\mathbb{A})}  \phi^{(s)}(  u g) \psi_\sca(u) \ du.
\]
\end{Thm}

The proof of this theorem is given in the next section.

\begin{Rem}\label{rem:eulerian fourier}
Note that $(N_{\scw}^{\heartsuit} \cap P,\psi_{\sca})$ defines a Fourier coefficient for $\GL_D(W^{\Delta,k})$ that lies in the orbit $(k^n)_D$. Thus the right-hand side is Eulerian for decomposable $\phi^{(s)}$. Indeed, let
\[
f^{(s)}(g)=f_{\phi}^{(s)}(g):=\int\limits_{[N_{\scw}^{\heartsuit} \cap P]} \phi^{(s)}( u  g)\psi_{\sca}(u) \ du.
\]
For fixed $g$, $\phi^{(s)}(g)\mapsto f^{(s)}(g)$ is a global $(k,n)_D$-functional for $\phi^{(s)}(g)\in \theta\cdot \nu^s$. By Lemma \ref{lem:eulerian}, for decomposable data $\phi^{(s)}=\otimes_v'\phi_v^{(s)}$, $f^{(s)}(g)$ is also decomposable:
\[
f^{(s)}(g)=\prod_v f_v^{(s)}(g_v).
\]
This implies that
\[
\begin{aligned}
&\int\limits_{N_{\scw}^{\heartsuit}(F) \cap P(F)\backslash N_{\scw}^{\heartsuit}(\mathbb{A})}  \phi^{(s)}(  u g) \psi_\sca(u) \ du \\
=&\int\limits_{N(W^{\nabla,k})(\ba)}\int\limits_{[N_{\scw}^{\heartsuit} \cap P]}  \phi^{(s)}(  u' u g) \psi_\sca(u'u) \ du' \ du\\
=&\prod_v \int\limits_{N(W^{\nabla,k})(F_v)}f_v^{(s)}(u_v g_v) \psi_\sca(u_v) \ du_v.\\
\end{aligned}
\]
\end{Rem}

\begin{Rem}
By analyzing the local integral in the Euler product, one can confirm special instances of \cite{Ginzburg06} Conjecture 5.13 regarding the nilpotent orbits attached to Eisenstein series on symplectic and split even orthogonal groups.
\end{Rem}

\subsection{Proof of Theorem \ref{thm:unfolding of fourier coefficient}}
\label{sec:unfolding of FC}
Recall that
\[
 E(\phi^{(s)})(g)=\sum_{\gamma \in P(F)\backslash G^{\square,k}(F)}\phi^{(s)}(\gamma g).
\]
When $\Re s\gg 0$,
\[
\begin{aligned}
&\int\limits_{[N_{\mathcal{W}}^{\heartsuit}]} E(\phi^{(s)})(ug) \psi_\sca(u) \ du\\
=&\int\limits_{[N_{\mathcal{W}}^{\heartsuit}]} \sum_{\gamma \in P(F)\backslash G^{\square,k}(F)} \phi^{(s)}(\gamma u g) \psi_\sca(u) \ du\\
=& \int\limits_{[N_{\mathcal{W}}^{\heartsuit}]} \sum_{\gamma \in P(F)\backslash G^{\square,k}(F)/N_{\mathcal{W}}^{\heartsuit}(F)} \
\sum_{\gamma'\in \gamma^{-1}P(F)\gamma\cap N_{\mathcal{W}}^{\heartsuit}(F)\bs N_{\mathcal{W}}^{\heartsuit}(F)}  \phi^{(s)}(\gamma \gamma' u g) \psi_\sca(u) \ du\\
=&\sum_{\gamma \in P(F)\backslash G^{\square,k}(F)/N_{\mathcal{W}}^{\heartsuit}(F)} \ \int\limits_{\gamma^{-1}P(F)\gamma\cap N_{\mathcal{W}}^{\heartsuit}(F)\backslash N_{\mathcal{W}}^{\heartsuit}(\mathbb{A})}  \phi^{(s)}(\gamma  u g) \psi_\sca(u) \ du.
\end{aligned}
\]
We now study the contribution for each $\gamma$.

Recall that $P\bs G^{\square,k}$ can be identified with the variety $\Omega(W^{\square,k})$ of maximal totally isotropic subspaces $W^{\Delta,k}$ via $P\gamma\mapsto W^{\Delta,k}\gamma$. If we write $L=W^{\Delta,k}\gamma$, then $P(L)=\gamma^{-1}P\gamma$. Given $\gamma$, write
\[
\begin{aligned}
I_{\gamma}(g)=&\int\limits_{\gamma^{-1}P(F)\gamma\cap N_{\mathcal{W}}^{\heartsuit}(F)\backslash N_{\mathcal{W}}^{\heartsuit}(\mathbb{A})}  \phi^{(s)}(\gamma  u g) \psi_\sca(u) \ du\\
=&\int\limits_{P(L)(F)\cap N_{\mathcal{W}}^{\heartsuit}(F)\backslash N_{\mathcal{W}}^{\heartsuit}(\mathbb{A})}  \phi^{(s)}(\gamma  u g) \psi_\sca(u) \ du.
\end{aligned}
\]
We now use results from Sect. \ref{sec:contribution in unfolding} to look for $\gamma$ such that $I_\gamma\neq 0$.

\begin{Lem}
If $\psi_{\sca}|_{[N(L)\cap N_{\mathcal{W}}^{\heartsuit}]}$ is nontrivial, then $I_\gamma(g)= 0$ for all $g\in G^{\square,k}(\ba)$.
\end{Lem}

\begin{proof}
Indeed, $I_\gamma(g)$ contains
\begin{equation}\label{eq:inner in first step in FC}
\int\limits_{[N(L)\cap N_{\mathcal{W}}^{\heartsuit}]}  \phi^{(s)}(\gamma  u g) \psi_\sca(u) \ du
\end{equation}
as an inner integral. As a function of $u\in [N(L)\cap N_{\mathcal{W}}^{\heartsuit}]$, $\phi^{(s)}(\gamma u g)$ is trivial. We conclude that \eqref{eq:inner in first step in FC} is zero. This implies that $I_\gamma(g)=0$ for all $g\in G^{\square,k}(\ba)$.
\end{proof}

The orbits such that $\psi_{\sca}|_{[N(L)\cap N_{\mathcal{W}}^{\heartsuit}]}$ is nontrivial can be eliminated. By Proposition \ref{prop:first vanishing criterion}, from now on we can assume that
\[
A_i(L\cap Y_i/L\cap Y_{i-1})\subset L\cap Y_{i+1}/L\cap Y_{i},\qquad i=1,\cdots, 2k-1.
\]
This gives a pair which defines a Fourier coefficient in $M(L)=\GL_D(L)$. The flag defining the unipotent subgroup is
\[
0\subset L\cap Y_1\subset  \cdots \subset L\cap Y_{k} \subset L\cap Y_{k-1}^\perp \subset\cdots \subset L\cap Y_1^{\perp}.
\]

\begin{Lem}\label{lem:get higher orbit}
If $L\cap W^{\nabla,k}\neq \{0\}$, then the pair $(N(L) \cap N_{\scw}^{\heartsuit}\bs P(L) \cap N_{\scw}^{\heartsuit},\psi_\sca)$ lies in an orbit higher than $(k^n)_D$.
\end{Lem}

\begin{proof}

Suppose that $L\cap W^{\nabla,k}\neq \{0\}$. Let  $i$ is the smallest index such that $L\cap Y_{i-1}=\{0\}$ and $L\cap Y_i\neq \{0\}$. The isomorphism $A_i:Y_i/Y_{i-1}\to Y_{i+1}/Y_i$ restricts to the map
\[
L\cap Y_i/L\cap Y_{i-1}\to L\cap Y_{i+1}/L\cap Y_i,
\]
which is an injection. In particular, we know that $L\cap Y_{i+1}/L\cap Y_i \neq \{0\}$. We now can repeat the same argument to obtain a sequence of injective maps
\[
L\cap Y_i\xrightarrow{A_i} L\cap Y_i\bs L\cap Y_{i+1} \xrightarrow{A_{i+1}}\cdots \to L\cap W^{\square,k}/L\cap Y_1^\perp.
\]
This sequence consists of at least $k$ injections and their composition is nonzero. Thus the coefficient is higher than $(k^n)_D$.
\end{proof}

Thus such orbits can be eliminated as $\theta$ is a representation of type $(k,n)_D$. We are left with orbits such that $L\cap W^{\nabla,k}=\{0\}$.
\begin{Lem}
If $L=W^{\Delta,k}\gamma$ satisfying $L\cap W^{\nabla,k}=\{0\}$, then $\gamma$ in the double coset $P\cdot N_{\scw}^{\heartsuit}$.
\end{Lem}

\begin{proof}
This is a consequence of Lemma \ref{lem:orbit under para and nil}. We can also argue directly. Since $W^{\Delta,k}\cap W^{\nabla,k}=\{0\}$ as well, there exists some $p\in P(W^{\nabla,k})$ such that $L=W^{\Delta,k}p$.
We write $p=mu$ where $m\in \GL_D(W^{\Delta,k}) \simeq \GL_D(W^{\nabla,k})$ and $u\in N(W^{\nabla,k})$. We have $m\in P$ and $u\in N_{\scw}^{\heartsuit}$. This completes the proof.
\end{proof}

We now know that only the double coset $P\cdot N_{\scw}^{\heartsuit}$ gives nonzero contribution.
We can choose $L=W^{\Delta,k}$ or $\gamma=1$ to calculate its contribution. Therefore,
\[
\int\limits_{[N_{\mathcal{W}}^{\heartsuit}]} E(\phi^{(s)})(ug) \psi_\sca(u) \ du=I_1(g),
\]
and the contribution is
\[
\int\limits_{N_{\scw}^{\heartsuit}(F)\cap P(F)\backslash N_{\scw}^{\heartsuit}(\mathbb{A})}  \phi^{(s)}(  u g) \psi_\sca(u) \ du.
\]
This completes the proof.

\section{The twisted doubling integrals}\label{sec:twisted doubling}

We now introduce the twisted doubling integrals.

\subsection{Fourier coefficients}\label{sec:FC in global integral}

We first construct a  Fourier coefficient associated with the orbit $((2k-1)^n 1^n)_D$.
We choose the following flag of totally isotropic subspaces in $W^{\square,k}$
\begin{equation}\label{eq:flag in twisted doubling}
0\subset W_k^\nabla \subset W_{k-1}^\nabla \oplus W_{k}^\nabla \subset \cdots \subset W_2^\nabla \oplus \cdots \oplus W_{k}^\nabla.
\end{equation}
Let $P_{\scw,k}^{\bullet}=M_{\scw,k}^{\bullet}\cdot N_{\scw,k}^{\bullet}$ be the corresponding parabolic subgroup. Then
\[
M_{\scw,k}^{\bullet}\simeq \GL_D(W_{k}^\nabla)\times \cdots \times \GL_D(W_{2}^\nabla) \times G(\scw_1^\square).
\]
When $k$ is clear in the context, we usually write drop the subscript.
The character is defined on the group $N_{\scw,k}^{\bullet}$.

We reindex the flag in \eqref{eq:flag in twisted doubling} as
\[
0\subset Y_1\subset \cdots \subset Y_{k-1}
\]
and extend it to
\[
0\subset  Y_1\subset \cdots \subset Y_{k-1} \subset Y_{k-1}^{\perp}\subset \cdots \subset Y_1^\perp\subset W^{\square,k}.
\]
We again adopt the convention in Sect. \ref{sec:FC of classical} Case 1 when discussing terms for $i\geq k$. Note that except $Y_{k-1}^\perp/Y_{k-1}=W_1^\square$, the quotient between two successive terms is isomorphic to either $W^\Delta$ or $W^\nabla$.

To describe the character, we have to specify elements
\[
A_i\in \Hom_D(Y_i/Y_{i-1},Y_{i+1}/Y_i)\simeq \End_D(W^\nabla),\qquad i=1,\cdots, k-2,
\]
and
\[
A_{k-1}\in \Hom_D(Y_{k-1}/Y_{k-2},Y_{k-1}^{\perp}/Y_{k-1})\simeq \Hom_D(W^\nabla,W^\square).
\]
We choose $A_1,\cdots,A_{k-2}$ to be the identity map in $\End_D(W^{\nabla})$. The map
\[
Y_{k-1}/Y_{k-2}\xrightarrow{A_{k-1}} Y_{k-1}^{\perp}/Y_{k-1}\xrightarrow{A_k} Y_{k-2}^{\perp}/Y_{k-1}^{\perp}
\]
is translated from
\[
W^\nabla\to W_+\oplus W_- \to W^\Delta,\qquad x^{\nabla}\mapsto (2x,0)\mapsto 2x^{\Delta}.
\]
We remind the reader that the identification here is not very important. What is important is that $A_{k}\circ A_{k-1}$ is an isomorphism.


It is straightforward to the check the pair $(N_{\scw}^{\bullet},\psi_{\scw}^{\bullet})$ gives a Fourier coefficient associated to the orbit $((2k-1)^{n}1^{n})_D$.

Given $(g_1,g_2)\in G\times G$, we define its action on $W_k^\nabla\oplus \cdots W_2^\nabla \oplus W_{1,+}\oplus W_{1,-}$ via
\[
(x_k^\nabla, x_{k-1}^\nabla,\cdots, x_{2}^\nabla, x_1^+, x_1^-)(g_1,g_2)=(x_k^\nabla g_1, x_{k-1}^\nabla ,\cdots, x_{2}^\nabla g_1, x_1^+ g_1, x_1^- g_2).
\]
This extends to an action of $G\times G$ on $W^{\square,k}$ and gives a map
\[
\iota:G\times G\to G^{\square,k}.
\]
It is in fact a homomorphism and in particular, the images of these two copies of $G$ commute in $G^{\square,k}$.
It is straightforward to check that $\iota(G\times G)$ lies in the stabilizer of $\psi_{\scw}^{\bullet}$ in $G^{\square,k}$.

For a subgroup $J$ of $G$, we define
\[
J^\diamondsuit=\{(g,g)\in G\times G\mid g\in J\}.
\]

\begin{Lem}
We have $\iota(G\times G)\cap P(W^{\Delta,k})=\iota(G^\diamondsuit)$.
\end{Lem}

\begin{proof}
Let $\iota(g_1,g_2)$ be such an element. Since it is in $P(W^{\Delta,k})$, we know
\[
(x_2^\Delta g_1, x_3^\Delta g_1,\cdots, x_{k}^\Delta g_1, x_1^+ g_1, x_1^- g_2)\in W^{\Delta,k}
\]
with $x_1^+=x_1^-$. This implies that $x_1^+g_1=x_1^+g_2$ for all $x_1^+\in W$. In other words $g_1=g_2$.
\end{proof}

\begin{Lem}\label{lem:modular quasi}
The modular quasicharacter $\delta_{\iota(G\times G);N_{\scw}^{\bullet}}(\iota(g_1,g_2))=1$ for any $g_1,g_2\in G$.
\end{Lem}

\begin{proof}
This follows from direct calculation using \cite{LR05} Lemma 1. Observe that $|\nu(\iota(g_1,g_2))|=1$ for any $(g_1,g_2)\in G\times G$.
\end{proof}

\subsection{The global integral}
Let $\pi$ be an irreducible cuspidal automorphic representation of $G(\ba)$ realized on a space $V_\pi\subset L^2(G(F)\bs G(\ba))$, where we fix an embedding $\pi\hookrightarrow V_\pi\subset \sca(G(\ba))$. The contragredient representation $\pi^\vee$ is realized on the complex conjugate $\overline{V_\pi}$ of $V_\pi$. The Petersson pairing $\scp=\scp_{\pi}: V_\pi\boxtimes \overline{V_\pi}\to \bc$ is defined by
\[
\scp_{\pi}(\xi_1\boxtimes \xi_2)=\int\limits_{G(F)\bs G(\ba)}\xi_1(g)\xi_2(g) \ dg.
\]

 For each pair of cusp forms $\xi_1\in V_\pi$ and $\xi_2\in \overline{V_\pi}$ and each section $\tilde\phi^{(s)}$ of $I(s,\theta)$, we consider the global zeta integral defined by
\begin{equation}\label{eq:global integral}
Z(\xi_1\boxtimes \xi_2,\phi^{(s)})=
\int\limits_{[G\times G]} \
\int\limits_{[N_{\scw}^{\bullet}]}
\chi_\theta(\nu(g_2))^{-1} \xi_1(g_1)\xi_2(g_2) E(\phi^{(s)})(u \cdot \iota(g_1,g_2))\psi_{\scw}^{\bullet}(u) \ du \ dg_1 \ dg_2.
\end{equation}
Since the two cusp forms are rapidly decreasing on $G(F)\bs G(\ba)$ and the Eisenstein series is only of moderate growth, we see that the integral converges absolutely for all $s$ away from the poles of the Eisenstein series and is hence meromorphic in $s$.

The main global identity is the following result.

\begin{Thm}\label{thm:statement of unfolding}
When $\Re s\gg 0$, $Z(\xi_1\boxtimes \xi_2,\phi^{(s)})$ equals
\[
\int\limits_{G^\diamondsuit(F)\bs (G\times G)(\ba)}\chi_\theta(\nu(g_2))^{-1} \xi_1(g_1)\xi_2(g_2) \int\limits_{(N_{\scw}^{\bullet}\cap P)(F)\bs N_{\scw}^{\bullet}(\ba)} \phi^{(s)}(u \cdot \iota(g_1,g_2))\psi_{\scw}^{\bullet}(u) \ du \ dg_1 \ dg_2.
\]
\end{Thm}

The proof of this theorem occupies the next section. We now explain that this integral is in fact an Euler product. As in Remark \ref{rem:eulerian fourier}, let
\[
f^{(s)}(g)=f_{\phi}^{(s)}(g):=\int\limits_{[N_{\scw}^{\bullet}\cap P]} \phi^{(s)}( u  g)\psi_{\scw}^{\bullet}(u) \ du.
\]
The pair $(N_{\scw}^{\bullet}\cap P,\psi_{\scw}^{\bullet})$ lies in the orbit $(k^n)_D$ and $\iota(G^{\diamondsuit})$ lies in the stabilizer of this pair. For fixed $g$, the function $u\mapsto \phi^{(s)}(ug)$ on $N_{\scw}^{\bullet}\cap P$ is an element of $\theta \cdot \nu^s$. Therefore, $f^{(s)}(g)$ is a $(k,n)_D$-Fourier coefficient of an element in $\theta \cdot \nu^s$.

Let $N_{\scw}^{\circ}=N_{\scw}^{\bullet}\cap N(W^{\nabla,k})$. Then for any $h\in G^{\square,k}(\ba)$,
\[
\begin{aligned}
 &\int\limits_{(N_{\scw}^{\bullet}\cap P)(F)\bs N_{\scw}^{\bullet}(\ba)} \phi^{(s)}( u \cdot  \iota(g_1,g_2)h)\psi_{\scw}^{\bullet}(u) \ du\\
 =&\int\limits_{N_{\scw}^\circ(\ba)} \ \int\limits_{[N_{\scw}^{\bullet}\cap P]} \phi^{(s)}(  u u' \cdot  \iota(g_1,g_2)h)\psi_{\scw}^{\bullet}( uu') \ du \ du'\\
 =&\int\limits_{N_{\scw}^\circ(\ba)}  f^{(s)}( u' \cdot  \iota(g_1,g_2)h)\psi_{\scw}^{\bullet}( u') \ du'\\
 =&\int\limits_{N_{\scw}^\circ(\ba)} f^{(s)}(\iota(g_2,g_2)\cdot u' \cdot  \iota(g_2^{-1}g_1,1)h)\psi_{\scw}^{\bullet}( u')  \ du'\\
  =&\chi_{\theta}(\nu(g_2))\int\limits_{N_{\scw}^\circ(\ba)} f^{(s)}(u \cdot  \iota(g_2^{-1}g_1,1)h)\psi_{\scw}^{\bullet}( u)  \ du.\\
 \end{aligned}
\]
Observe that we use change of variable and Lemma \ref{lem:modular quasi} in the third equality, and Lemma \ref{lem:invariance I global} in the last equality. As a consequence, we can write $Z(\xi_1\boxtimes \xi_2,\phi^{(s)})$ as
\[
\begin{aligned}
 &\int\limits_{G^{\diamondsuit}(F)\bs (G\times G)(\ba)}\xi_1(g_1)\xi_2(g_2)\int\limits_{N_{\scw}^\circ(\ba)} f^{(s)}(u \cdot  \iota(g_2^{-1}g_1,1))\psi_{\scw}^{\bullet}( u)  \ du \ dg_2 \ dg_1\\
 =&\int\limits_{G^{\diamondsuit}(F)\bs (G\times G)(\ba)}\xi_1(g_2g_1)\xi_2(g_2)\int\limits_{N_{\scw}^\circ(\ba)} f^{(s)}(u \cdot  \iota(g_1,1))\psi_{\scw}^{\bullet}( u)  \ du\ dg_2 \ dg_1\\
 =&\int\limits_{G(\ba)} \ \int\limits_{[G^{\diamondsuit}]}\xi_1(g_2g_1)\xi_2(g_2)\int\limits_{N_{\scw}^\circ(\ba)} f^{(s)}(u \cdot  \iota(g_1,1))\psi_{\scw}^{\bullet}( u)  \ du\ dg_2 \ dg_1\\
  =&\int\limits_{G(\ba)}\scp(\pi(g)\xi_1\boxtimes \xi_2) \int\limits_{N_{\scw}^\circ(\ba)} f^{(s)}(u  \cdot  \iota(g,1))\psi_{\scw}^{\bullet}( u)  \ du  \ d g.\\
 \end{aligned}
\]
For fixed $g$, $\phi^{(s)}(g)\mapsto f^{(s)}(g)$ is a global $(k,n)_D$-functional for $\phi^{(s)}(g)\in \theta\cdot \nu^s$. By Lemma \ref{lem:eulerian}, for decomposable data $\phi^{(s)}=\otimes_v'\phi_v^{(s)}$, $f^{(s)}(g)$ is also decomposable:
\[
f^{(s)}(g)=\prod_v f_v^{(s)}(g_v).
\]
If we assume further that $\xi_i=\otimes_v' \xi_{i,v}$ for $i=1, 2$, then
\[
Z(\xi_1\boxtimes \xi_2,\phi^{(s)})=\prod_v \int\limits_{G(F_v)}\scp_v(\pi(g_v)\xi_{1,v}\boxtimes \xi_{2,v}) \int\limits_{N_{\scw}^\circ(F_v)} f_v^{(s)}(u_v  \cdot  \iota(g_v,1))\psi_{\scw}^{\bullet}( u_v)  \ du_v  \ d g_v.
\]

\begin{Rem}
In \cite{CFGK19} and \cite{CFK}, the generalized Speh representation is used as $\theta$. In this case, the unramified calculation in these two papers shows that the global integral represents the tensor product $L$-function. One can in fact choose other representations of type $(k,n)_D$ as the inducing data. For instance, one can choose the isobaric sum of several generalized Speh representations. In this case, the global integral represents a product of tensor product $L$-functions.
\end{Rem}

\section{Unfolding}\label{sec:unfolding}

The goal in this section is to prove Theorem \ref{thm:statement of unfolding}.

 When $\Re s \gg 0$, the global integral becomes
\[
\begin{aligned}
&Z(\xi_1\boxtimes \xi_2,\phi^{(s)})\\
=&\int\limits_{[G\times G]}
\chi_\theta(\nu(g_2))^{-1} \xi_1(g_1)\xi_2(g_2) \ \int\limits_{[N_{\scw}^{\bullet}]} \sum_{\gamma\in P(F)\bs G^{\square,k}(F)} \phi^{(s)}(\gamma u \cdot \iota(g_1,g_2))\psi_{\scw}^{\bullet}(u) \ du \ dg_1 \ dg_2.\\
\end{aligned}
\]
We can rewrite the integral as a sum over $P(F)\bs G^{\square,k}(F)/\iota(G\times G)N_{\scw}^{\bullet}(F)$. Our goal is to show that, only the double coset $P(F)\iota(G\times G)N_{\scw}^{\bullet}(F)$ supports nonzero contribution.

We proceed in a slightly different way. We have
\[
\begin{aligned}
& \int\limits_{[N_{\scw}^{\bullet}]} \ \sum_{\gamma\in P(F)\bs G^{\square,k}(F)}  \phi^{(s)}(\gamma u \cdot \iota(g_1,g_2))\psi_{\scw}^{\bullet}(u) \ du\\
=&  \int\limits_{[N_{\scw}^{\bullet}]} \ \sum_{\gamma\in P(F)\bs G^{\square,k}(F)/N_{\scw}^{\bullet}(F)} \ \sum_{\gamma'\in  N_{\scw}^{\bullet}(F) \cap \gamma^{-1}P(F)\gamma\bs N_{\scw}^{\bullet}(F)}  \phi^{(s)}(\gamma \gamma' u \cdot \iota(g_1,g_2))\psi_{\scw}^{\bullet}(u) \ du\\
=& \sum_{\gamma\in P(F)\bs G^{\square,k}(F)/N_{\scw}^{\bullet}(F)} \ \int\limits_{(N_{\scw}^{\bullet}\cap \gamma^{-1}P\gamma)(F)\bs N_{\scw}^{\bullet}(\ba)} \phi^{(s)}(\gamma  u \cdot \iota(g_1,g_2))\psi_{\scw}^{\bullet}(u) \ du.
\end{aligned}
\]
For each $\gamma\in P(F)\bs G^{\square,k}(F)/N_{\scw}^{\bullet}(F)$ and $h\in G^{\square,k}(\ba)$, we write
\[
\begin{aligned}
I_\gamma(h)=& \int\limits_{(N_{\scw}^{\bullet}\cap \gamma^{-1}P\gamma)(F)\bs N_{\scw}^{\bullet}(\ba)} \phi^{(s)}(\gamma  u  h)\psi_{\scw}^{\bullet}(u) \ du,\\
J_\gamma(h)=& \int\limits_{[N_{\scw}^{\bullet}\cap \gamma^{-1}P\gamma]} \phi^{(s)}(\gamma  u h)\psi_{\scw}^{\bullet}(u) \ du.
\end{aligned}
\]
Thus
\[
I_\gamma(h)= \int\limits_{(N_{\scw}^{\bullet}\cap \gamma^{-1}P\gamma\bs N_{\scw}^{\bullet})(\ba)} J_\gamma(uh)\psi_{\scw}^{\bullet}(u) \ du.
\]

Recall that $P\bs G^{\square,k}$ can be identified with the variety $\Omega(W^{\square,k})$ of maximal totally isotropic subspaces of $W^{\square,k}$. The identification is given by $\gamma\mapsto  W^{\Delta,k}\gamma$. The group $\gamma^{-1} N(W^{\Delta,k}) \gamma$ is the unipotent radical of the parabolic subgroup $P(W^{\Delta,k}\gamma)$. Let $L= W^{\Delta,k}\gamma$.


We now look for $\gamma$ such that $I_{\gamma}\neq 0$.
We now give an outline of the proof.

\begin{enumerate}
\item We first proceed as in Sect. \ref{sec:unfolding of FC} and use the character $\psi_{\scw}^\bullet$ and the orbit of $\theta$ to eliminate some orbits. Our first conclusion is that the remaining orbits satisfy
\begin{equation}\label{eq:unfolding empty intersection}
L \cap (W_2^\nabla\oplus \cdots \oplus W_k^\nabla)=\{0\}.
\end{equation}
\item We now assume that \eqref{eq:unfolding empty intersection} holds. For the remaining $\gamma$, we consider the orbit under the action of $\iota(G\times G)(F)$. We define an invariant $\kappa(L)$ which is similar to \cite{PSR87} Sect. 2. This invariant characterizes the orbit.
\item We now proceed as in \cite{PSR87} to eliminate the case $\kappa(L)>0$. There remains only $P(F)\cdot\iota(G\times G)N_{\scw}^{\bullet}(F)$ with nonzero contribution.
\end{enumerate}

In the rest of the proof, we keep using the notation $Y_i,A_i$ etc. from Sect. \ref{sec:FC in global integral}.
\subsection{First steps}

We already know that if the restriction of $\psi_{\scw}^{\bullet}$ to $[N(L) \cap N_{\scw}^{\bullet}]$ is nontrivial, then $J_\gamma(g)=0$ and consequently $I_\gamma(g)=0$. Thus we can assume that the restriction of $\psi_{\scw}^{\bullet}$ to $[N(L) \cap N_{\scw}^{\bullet}]$ is trivial. By Proposition \ref{prop:first vanishing criterion}, this is equivalent to
\[
A_i(L\cap Y_i/L\cap Y_{i-1})\subset L\cap Y_{i+1}/L\cap Y_{i},\qquad i\geq 1.
\]
This induces a character on $(N(L) \cap N_{\scw}^{\bullet} )\bs (P(L) \cap N_{\scw}^{\bullet})$. This subgroup can be viewed as a subgroup of $M(L)$. It is the unipotent radical of the parabolic subgroup of $\GL_D(L)$ stabilizing
\[
0\subset L\cap Y_1\subset \cdots \subset L\cap Y_{k-1}\subset L\cap Y_{k-1}^\perp \subset \cdots \subset L\cap Y_1^{\perp}\subset L.
\]
Thus $J_\gamma(g)$ contains
\[
\int\limits_{[N(L) \cap N_{\scw}^{\bullet} \bs P(L)\cap N_{\scw}^{\bullet}]} \phi^{(s)}(\gamma u g))\psi_{\scw}^{\bullet}(u) \ du
\]
as an inner integral.

\begin{Lem}
If $L\cap Y_{k-1} \neq \{0\}$, then the pair $(N(L) \cap N_{\scw}^{\bullet}\bs P(L) \cap N_{\scw}^{\bullet},\psi_{\scw}^{\bullet})$ lies in an orbit higher than $(k^n)_D$. Consequently, $J_{\gamma}(g)=0$.
\end{Lem}

\begin{proof}
The proof is identical to the proof of Lemma \ref{lem:get higher orbit}. Note that we need to use the fact that $A_{k}\circ A_{k-1}$ is an isomorphism.
\end{proof}

\subsection{The invariant $\kappa(L)$}\label{sec:invariant}

Define
\[
\tilde\Omega(W^{\square,k})=\{L\in\Omega(W^{\square,k})\mid L\cap Y_{k-1}=\{0\}\}.
\]
This subset corresponds to the open subset $P\bs P\cdot P_{\scw}^{\bullet}\subset P\bs G^{\square,k}$. 

\begin{Lem}\label{lem:dimension counting}
Let $L\in \tilde\Omega(W^{\square,k})$. Then
\[
\dim (L\cap Y_{k-i}^\perp)=i\cdot \dim W.
\]
\end{Lem}

\begin{proof}
A simple dimension counting shows that
\[
\dim(L\cap Y_{k-i}^\perp)\geq \dim(L)+\dim(Y_{k-i}^\perp)-\dim(W^{\square,k})=i\cdot \dim(W).
\]

We now note that $L\cap Y_{k-i}^\perp \to Y_{k-i}^\perp/ Y_{k-i}$
is an injection and can be viewed as a totally isotropic subspace of $Y_{k-i}^\perp/ Y_{k-i}$. This proves the reverse part.
\end{proof}

From now on, we assume that $L\in \tilde\Omega(W^{\square,k})$.
Define $\bar L$ to be the image of the injection
\[
L\cap Y_{k-1}^\perp\to Y_{k-1}^\perp/Y_{k-1}\to W_1^\square.
\]
Observe that $\bar L$ is a totally isotropic subspace of $W_1^{\square}$ and by Lemma \ref{lem:dimension counting}, it is maximal.
Let
\[
L^+= L\cap (W_{1,+}\oplus Y_{k-1}),\qquad L^- =L\cap (W_{1,-}\oplus Y_{k-1}).
\]
Let $\bar L^+$ be the image of the injection
\[
L^+\to (W_{1,+}\oplus Y_{k-1})/Y_{k-1}\to W_{1,+},
\]
and similar for $\bar L^-$. Observe that $\bar L^{\pm}=\bar L\cap W_{1,\pm}$.
Define
\[
\kappa^+(L)=\dim (\bar L^+ ),\qquad \kappa^-(L)=\dim (\bar L^- )
\]
Both $\kappa^+(L)$ and $\kappa^-(L)$ are less than or equal to $\dim W$. It is straightforward to verify that both of them are invariant under $\iota(G\times G) N_{\scw}^{\bullet}(F)$.

\begin{Lem}
We have that $\kappa^+(L)=\kappa^-(L)$.
\end{Lem}

\begin{proof}
The proof is essentially the same as \cite{PSR87} Lemma 2.1.
Consider the projection
\[
\pi^+: W_1^\square  \to W_{1,+},\qquad
\pi^-: W_1^\square  \to W_{1,-}.
\]
Then
\[
\bar L^{\mp}=\mathrm{Ker}(\pi^{\pm}|_{\bar L}).
\]
Let $\bar L'=\pi^+(\bar L)$ and $\bar L''=\pi^-(\bar L)$. Then
\[
\dim W=\dim \bar L=\dim \bar L'+\dim \bar L^-=\dim\bar  L''+\dim\bar  L^+.
\]

On the other hand, $\bar L^+\subset \bar L'$. We note that $\la \bar L^+,\bar L^{-}\ra^{\square}\subset \la W_{1,+},W_{1-}\ra^{\square}=0$, and $\la\bar  L^+,\bar L\ra^{\square}\subset \la \bar L,\bar L\ra^\square=0$. This implies that $\la\bar L^+,\bar L'\ra^{\square}=0$ as well.  
One can also check that $\bar L^+$ is the kernel of the form $\la \ , \ \ra$ restricted to $\bar L'$. In fact, if $x\in \bar L'$ satisfies $\la x,\bar L'\ra=0$, then $\la x,\bar L'+\bar L^-\ra^{\square}=\la x,\bar L\ra^{\square}=0$. This implies that $x\in \bar L$ and consequently $x\in \bar L^+$.

Since the form is non-degenerate on $W_{1,+}$, there must be a subspace $L_+\subset W_{1,+}$, of the same dimension as $\bar L^+$, which pairs non-degenerately with $\bar L^+$.  It is easy to check that $\bar L'\cap L_+=\{0\}$ and thus $\bar L'\oplus L_+\subset W_{1,+}$. This implies that
\[
\dim \bar L'+\dim L_+\leq \dim W.
\]
Thus we conclude that  $\dim \bar L^+\leq \dim \bar L^-$. The other direction can be proved similarly. Thus $\dim \bar  L^+=\dim\bar  L^-$.
\end{proof}

From now on,  we  simply write $\kappa(L)=\kappa^+(L)=\kappa^-(L)$.

\begin{Lem}
Let $L,M\in \tilde\Omega(W^{\square,k})$.
If $\kappa(L)=\kappa(M)$, then there exists $g\in \iota(G\times G)N_{\scw}^{\bullet}(F)$ such that $Lg=M$.
\end{Lem}

\begin{proof}
Let $\bar \iota:G\times G\to G(\scw_1^{\square})$ defined as in the original doubling method. This is the same as $\iota$ when $k=1$. From \cite{PSR87} Sect. 2 and 4, we know that there exist $g\in G(F)\times G(F)$ such that $\bar L \cdot \bar \iota(g)=\bar M$.

We can replace $L$ by $L\cdot \iota(g)$ and without loss of generality, we assume $\bar L=\bar M$. This implies that
\[
L\cap Y_{k-1}^{\perp}=M\cap Y_{k-1}^{\perp}.
\]
We already know that $L\cap Y_{i}=M\cap Y_{i}=\{0\}$ for $i=1,\cdots, k-1$. By Lemma \ref{lem:orbit under para and nil}, there exists $u\in N_{\scw}^{\bullet}(F)$ such that $Lu=M$. This completes the proof.
\end{proof}

Given $L\in \tilde\Omega(W^{\square,k})$, we now construct a slightly different one which is more convenient for future consideration. Recall that $\bar L\subset W_1^\square$. We set
\[
\tilde{L}=\bar L\oplus W_2^\Delta \oplus \cdots \oplus W_k^\Delta.
\]
It is straightforward to check that $\tilde{L}\in \tilde\Omega(W^{\square,k})$, and $\kappa(\tilde{L})=\kappa(L)$.

We now assume that $\kappa(L)>0$. Let $P^+$ be the parabolic subgroup of $G$ preserving the flag $\bar L^+\subset \bar L'\subset W$ and $P^-$ be the parabolic subgroup of $G$ preserving the flag $\bar  L^-\subset \bar  L''\subset W$. Since $\kappa(L)>0$, these are proper parabolic subgroups. Let $N^{\pm}$ be the unipotent radical of $P^{\pm}$ so that $N^+\times N^-$ is the unipotent radical of the proper parabolic subgroup $P^+\times P^-$ of $G\times G$.
The results in \cite{PSR87} show that $\bar \iota(N^+\times N^-)$ lies in the stabilizer of $\bar L$ in $\bar \iota(G\times G)$. It is now easy to check that $\iota(N^+\times N^-)$ lies in the stabilizer of $\tilde L$ (but not necessarily $L$) in $\iota(G\times G)$. Consequently, $\iota(N^+\times N^-)$ is in the stabilizer of $(P(\tilde L)\cap N_{\scw}^{\bullet},\psi_{\scw}^{\bullet})$.

\subsection{The case $\kappa(L)>0$}
We now have shown that, when $\Re s \gg 0$,  the global integral $Z(\xi_1\boxtimes \xi_2,\phi^{(s)})$ equals
\begin{equation}\label{eq:formula 1 last step}
\int\limits_{[G\times G]}\chi_\theta(\nu(g_2))^{-1} \xi_1(g_1)\xi_2(g_2)\sum_{\gamma \in P(F)\bs P(F)\cdot P_{\scw}^{\bullet}(F)/N_{\scw}^{\bullet}(F)} I_\gamma(\iota(g_1,g_2)) \ d g_1 \ d g_2.
\end{equation}

We now exchange the sum and integration in \eqref{eq:formula 1 last step}. This shows that \eqref{eq:formula 1 last step} equals
\[
\sum_{\gamma \in P(F)\bs P(F)\cdot P_{\scw}^{\bullet}(F)/\iota(G\times G)N_{\scw}^{\bullet}(F)} K_\gamma,
\]
where
\[
\begin{aligned}
K_\gamma=& \sum_{\gamma'\in \iota^{-1}(\gamma^{-1}P(F)\gamma\cap \iota(G\times G)(F))\bs (G\times G)(F)}\ \int\limits_{[G\times G]}\chi_\theta(\nu(g_2))^{-1} \xi_1(g_1)\xi_2(g_2) I_{\gamma\gamma'}(\iota(g_1,g_2))\ d g_1 \ d g_2\\
=&\int\limits_{\iota^{-1}(P(L)\cap \iota(G\times G))(F)\bs (G\times G)(\ba)}\chi_\theta(\nu(g_2))^{-1} \xi_1(g_1)\xi_2(g_2) I_\gamma(\iota(g_1,g_2))\ d g_1 \ d g_2.\\
\end{aligned}
\]
The outer sum is indeed indexed by the invariant $\kappa(L)$. To calculate $K_\gamma$,
we can choose $L$ as in the previous section so that $\iota(N^+\times N^-)\subset P(L)$.

\begin{Lem}
If $L\cap Y_{k-1} = \{0\}$,
then  $N(L)\cap N_{\scw}^{\bullet}=\{1\}$.
\end{Lem}

\begin{proof}
Let $u\in N(L)\cap N_{\scw}^{\bullet}$. Then  $(u-1)(W^{\square,k}) \subset L\cap (W_2^\nabla \oplus \cdots \oplus W_k^\nabla)$. This implies that $u-1=0$.
\end{proof}

Thus we can view $P(L)\cap N_{\scw}^{\bullet}$ as a subgroup of $M(L)\simeq \GL_D(L)$. 

We next show that the Fourier coefficient for $\GL_D(L)$  that appears in $I_{\gamma}(\iota(g_1,g_2))$ is of the form that is considered in Sect. \ref{sec:invariance under stab}. 
We use Proposition \ref{prop:invariance} to show that $I_\gamma(\iota(g_1,g_2))$ is left-invariant under the subgroup $\iota(\{1\}\times [N^-])$. Let $n_0$ be the Witt index of $\scw$. Recall that $\kappa(L)\leq n_0$.

\begin{Lem}
For $\iota(1,g_2)\in \iota([N^+ \times N^-])\subset [\iota(G\times G)\cap P(L)]$ and $h\in G^{\square,k}(\ba)$,
\[
J_\gamma(\iota(1,g_2)h)=J_\gamma(h).
\]
\end{Lem}

\begin{proof}
 We show that the Fourier coefficient that is used in $J_\gamma$ is the one considered in Sect. \ref{sec:invariance under stab}. The group $P(L)\cap N_{\scw}^{\bullet}$ stabilizes both $L$ and the flag
\[
0\subset Y_1 \subset \cdots \subset Y_{k-1} \subset Y_{k-1}^{\perp}\subset \cdots \subset Y_1^{\perp}\subset W^{\square,k}.
\]
Since  $L\cap Y_{k-1}=\{0\}$, this implies that it stabilizes
\[
0\subset L\cap Y_{k-1}^\perp \subset \cdots \subset L\cap  Y_1^\perp\subset L.
\]
By Lemma \ref{lem:dimension counting}, $\dim(L\cap Y_{k-i}^\perp)=i\cdot \dim W$.
We now describe the character. Note that
\[
L\cap Y_{i-1}^{\perp}/L\cap Y_{i}^{\perp}\cong Y_{i-1}^{\perp}/Y_i^{\perp},\qquad i=2, \cdots, k-2.
\]
Thus the map
\[
L\cap Y_{i}^{\perp}/L\cap Y_{i+1}^{\perp}\to L\cap Y_{i-1}^{\perp}/L\cap Y_{i}^{\perp},\qquad i=k-2,\cdots, 1
\]
is simply the same as $Y_{i}^{\perp}/Y_{i+1}^{\perp}\to Y_{i-1}^{\perp}/Y_{i}^{\perp}$ and is again an isomorphism.

It remains to determine the map
\begin{equation}\label{eq:restriction of the first map}
L\cap Y_{k-1}^{\perp}/L\cap Y_{k-1}\to L\cap Y_{k-2}^{\perp}/L\cap Y_{k-1}^{\perp}
\end{equation}
which comes from
$A_{k}:Y_{k-1}^{\perp}/Y_{k-1}\to Y_{k-2}^{\perp}/Y_{k-1}^{\perp}$ via restriction.
Recall that $A_k$ is translated from
\[
W_{1,+}\oplus W_{1,-}\to W^\Delta,\qquad (x,y)\mapsto 2x^\Delta.
\]
Therefore, $\bar L^-$ is in the kernel of \eqref{eq:restriction of the first map}, and $S_{\sca}$ in Proposition \ref{prop:invariance} is $\iota(\{1\}\times N^-)$. The rank of this map is at least $n-n_0>0$.  By Proposition \ref{prop:invariance}, the coefficient $J_\gamma(h)$ is invariant under $\iota(\{1\}\times [N^-])$.
\end{proof}

\begin{Prop}\label{prop:invariance property of integral}
For $\iota(1,g_2)\in \iota(\{1\}\times [N^-])\subset [\iota(G\times G)\cap P(L)]$ and $h\in G^{\square,k}(\ba)$,
\[
I_\gamma(\iota(1,g_2)h)=I_\gamma(h).
\]
\end{Prop}

\begin{proof}
We have
\[
\begin{aligned}
I_\gamma(\iota(1,g_2)h)= & \int\limits_{(N_{\scw}^{\bullet}\cap \gamma^{-1}P\gamma\bs N_{\scw}^{\bullet})(\ba)}J_\gamma(u\cdot \iota(1,g_2)h)\psi_{\scw}^{\bullet}(u) \ du\\
=& \int\limits_{(N_{\scw}^{\bullet}\cap \gamma^{-1}P\gamma\bs N_{\scw}^{\bullet})(\ba)}J_\gamma(\iota(1,g_2)\iota(1,g_2^{-1})u\cdot \iota(1,g_2)h)\psi_{\scw}^{\bullet}(u) \ du\\
=& \int\limits_{(N_{\scw}^{\bullet}\cap \gamma^{-1}P\gamma\bs N_{\scw}^{\bullet})(\ba)}J_\gamma(\iota(1,g_2^{-1})u\cdot \iota(1,g_2)h)\psi_{\scw}^{\bullet}(u) \ du\\
=& \int\limits_{(N_{\scw}^{\bullet}\cap \gamma^{-1}P\gamma\bs N_{\scw}^{\bullet})(\ba)}J_\gamma(uh)\psi_{\scw}^{\bullet}(u) \ du=I_\gamma(h).\\
\end{aligned}
\]
The last step is obtained by a change of variables. Note that $\iota(G\times G)\cap P(L)$ is in the stabilizer of both $N_{\scw}^{\bullet}\cap \gamma^{-1}P\gamma\bs N_{\scw}^{\bullet}$ and $\psi_{\scw}^{\bullet}$, and by Lemma \ref{lem:modular quasi}, the modular quasicharacter incurred $\delta_{\iota(G\times G)\cap P(L):N_{\scw}^{\bullet}}(\iota(\{1\}\times N^-))=1$.
\end{proof}

Suppose $\kappa(L)>0$, then $N^-$ is nontrivial. By Proposition \ref{prop:invariance property of integral}, the integral $K_\gamma$ contains the following inner integral
\[
\int\limits_{[N^-]}\xi_2(ug_2) \ du.
\]
As $\xi_2$ is a cusp form, this implies that $K_\gamma=0$.

\subsection{Finishing the calculation}
We are now ready to finish the unfolding process. We only have consider $L\in \Omega(W^{\square,k})$ such that $L\cap Y_{k-1}=\{0\}$. For such subspaces, one can define an invariant $\kappa(L)$. This invariant characterizes the $\iota(G\times G)N_{\scw}^{\bullet}(F)$-orbit of such $L$.

It remains to consider the case $\kappa(L)=0$. In this case, we can simply choose $L=W^{\Delta,k}$ or $\gamma=1$ to calculate the contribution. Therefore, we have shown that $Z(\xi_1\boxtimes \xi_2,\phi^{(s)})=K_1$. Note that $\iota(G\times G)\cap P=\iota(G^{\diamondsuit})$ and we indeed have
\[
Z(\xi_1\boxtimes \xi_2,\phi^{(s)})= \int\limits_{G^\diamondsuit(F)\bs (G\times G)(\ba)}\chi_\theta(\nu(g_2))^{-1} \xi_1(g_1)\xi_2(g_2) I_1(\iota(g_1,g_2))\ d g_1 \ d g_2.
\]
This completes the proof.

\section{Some local results}\label{sec:local results}

In this section, we state some local results whose proofs are analogous to those global ones. These results are necessary ingredients in order to develop the local theory of the twisted doubling integrals. In this section, we assume that $F$ is a non-Archimedean local field. The Archimedean case requires different tools and will be treated in a follow-up article.

The local analogue of unfolding uses the Geometric Lemma of \cite{BZ77}. This requires that the number of double cosets involved is finite. However, for the cases we treat in this paper, the number of double cosets is infinite. To give a careful treatment, we use the localization principle of Bernstein.

Notation: for an algebraic variety $X$ over $F$, we usually write $X$ for the set of rational points $X(F)$.



\subsection{Preliminaries on distributions}

We first recall some necessary tools in the theory of distributions. The main reference is \cite{BZ76}. We also use \cite{BZ77} and \cite{Bump97} Sect. 4.3.

\begin{Def}
We call a topological space an $l$-space if it is Hausdorff, locally compact, and totally disconnected. We call a topological group $G$ an $l$-group if it is an $l$-space.
\end{Def}

We always assume that an $l$-group $G$ is countable at infinity, that is, $G$ is a union of countably many compact subsets.

For an $l$-space $X$, let $C^\infty(X)$ be the space of locally constant functions on $X$ and $\crs(X)$ be the space of locally constant, compactly supported functions on $X$. Let $\crc$ be the sheaf of rings of locally constant functions on $X$.

\begin{Def}
By an $l$-sheaf $\scf$ on an $l$-space $X$ we mean an arbitrary sheaf of modules over the sheaf of rings $\crc$.
\end{Def}

For an $l$-space $X$, let $\Sh(X)$ denotes the category of $l$-sheaves on $X$. For an $l$-group $G$, let $\Alg(G)$ denotes the category of smooth representations of $G$.

Let $\scf$ be an $l$-sheaf on $X$. Let $\scf(X)$ denote the space of sections of $\scf$ over $X$ and by $\scf_c(X)$ the subspace of sections with a compact support. Both are modules over the ring $C^{\infty}(X)$ and hence over $\crs(X)$.

\begin{Prop}[\cite{BZ76} 1.14]
The functor $\scf\mapsto \scf_c(X)$ is an equivalence of the category $\Sh(X)$ with the category of all $\crs(X)$-modules $M$ satisfying the condition $\crs(X)\cdot M=M$.
\end{Prop}

We now recall the notion of distributions.

\begin{Def}
A distribution $\scd$ on an $l$-sheaf $\scf$ is a linear functional on $\scf_c(X)$. Let $\fd(X,\scf)$ denote the space of distributions on $\scf$.
\end{Def}

Let $Z$ be a closed subset of $X$. The restriction of the sheaf $\scf$ to $Z$ is denoted $\rres_Z(\scf)$. Let $\scf_c(Z)=(\rres_Z(\scf))_c(Z)$. For $x\in X$, let $\scf_x$ denotes the stalk of $\scf$ at $x$. Notice that $\scf_c(\{x\})=\scf_x$.

We now discuss induced representations in the language of $l$-sheaves. The following material can be found in \cite{BZ76} Sect. 2.21 -- 2.23. We first introduce the notion of $G$-sheaf.

\begin{Def}
We say that $\gamma:G\to \Aut(X,\scf)$ is an action of $G$ on a pair $(X,\scf)$ if the action of $G$ on $X$ is continuous and the representation of $G$ on $\scf_c(X)$ is smooth.
\end{Def}

\begin{Def}
Fix a continuous action $\gamma_0$ of $G$ on $X$. Let $\Sh(X,G)$ be the category of $G$-sheaves on $X$. An object of $\Sh(X,G)$ is an $l$-sheaf $\scf\in \Sh(X)$ with an action $\gamma$ of $G$ on $(X,\scf)$ such that the restriction of $\gamma$ on $X$ is $\gamma_0$. By morphisms in $\Sh(X,G)$ we mean $G$-equivariant morphisms of sheaves on $X$.
\end{Def}

The representation of $G$ on $\scf_c(X)$ induces an action of $G$ on $\fd(X,\scf)$, which is denoted by $g\cdot \scd$ for $g\in G$.

Assume that the action $\gamma_0$ of $G$ on $X$ is transitive. As $G$ is countable at infinity then $X$ is homeomorphic to the quotient space $H\bs G$, where $H$ is the stabilizer of some point $x\in X$ (\cite{BZ76} 1.5). The restriction gives a functor
\begin{equation}\label{eq:restriction G equivariant}
\Sh(X,G)\to \Sh(\{x\},H)=\Alg(H).
\end{equation}

Conversely, let $(\pi,V)$ be a smooth representation of $H$. The space of compact induction  $\ind_H^G(\pi)$ is a $\crs(H\bs G)$-module, and therefore corresponds to a $G$-sheaf $\scf^\pi$. This gives a functor $\Alg(H)\to \Sh(X,G)$. By \cite{BZ76} 2.23, this functor is an inverse to \eqref{eq:restriction G equivariant} and $\Sh(X,G)$ is equivalent to $\Alg(H)$ as categories.

\begin{Rem}\label{rem:stalks}
Let $x'=xg\in X$ for some $g\in G$. Then the stalk $\scf_{x'}$ is the representation $g^{-1}(\pi)$ of $g^{-1}Hg$.
\end{Rem}

Let $Q$ be a closed subgroup of $G$ and $Z$ be a closed $Q$-invariant subset of $X$. Assume that the action of $Q$ on $Z$ is transitive. Then $Z$ is of the form $H\bs HgQ$ for some $g\in G$. Thus $Z$ is homeomorphic to $H\bs HgQ\simeq g^{-1}Hg\cap Q\bs Q$. We would like to consider the composition of the following functors
\begin{equation}\label{eq:composition}
\Alg(H)\xrightarrow{\ind} \Sh(X,G)\xrightarrow{\res} \Sh(Z,Q)\xrightarrow{\mathrm{sec}} \Alg(Q).
\end{equation}
Clearly, the resulting representation is an induced representation on $Q$ with inducing data on $g^{-1}Hg\cap Q\bs Q$. It suffices to know the inducing data.



\begin{Prop}\label{prop:composition of functors}
The resulting representation in \eqref{eq:composition} is $\ind_{g^{-1}Hg\cap Q}^Q (g^{-1}(\pi))$.
\end{Prop}

\begin{proof}
It suffices to check the inducing data, or equivalently the functor
\[
\Alg(H)\xrightarrow{\ind} \Sh(X,G)\xrightarrow{\res} \Sh(Z,Q) \xrightarrow{\res}\Sh(\{xg\}).
\]
The second restriction functor is simply taking the stalk. The restriction functor from $X$ to $Z$ preserves stalks and therefore by Remark \ref{rem:stalks}, the inducing data is $\scf_{xg}=g^{-1}(\pi)$.
\end{proof}

Finally, we recall the localization principle of Bernstein (\cite{Bernstein84} Sect. 1.4). The version we use here is stated in \cite{Bump97} Proposition 4.3.15.

\begin{Thm}\label{thm:localization principle}
Let $X$ and $Y$ be $l$-spaces, and let $p:X\to Y$ be a continuous map. Let $\scf$ be an $l$-sheaf on $X$. Suppose that $G$ is a group acting on $(X,\scf)$. Assume that the action satisfies $p(xg)=p(x)$ for $g\in G, x\in X$. Let $\chi$ be a character of $G$.

Assume that there are no nonzero distributions $\scd$ in $\fd(p^{-1}(y),\scf_{p^{-1}(y)})$ that
\begin{equation}\label{eq:equi condition}
g\cdot \scd=\chi(g)\scd,\qquad (g\in G)
\end{equation}
for any $y\in Y$. Then there are no nonzero distributions in $\mathfrak{D}(X,\scf)$ satisfying \eqref{eq:equi condition}.
\end{Thm}

\subsection{Results}

From now on, $D$ can be any of the five types given in Sect. \ref{sec:groups}. The representation $\theta$ is assumed to be of type $(k,n)_D$.

\begin{Prop}\label{prop:unfolding FC local}
We use the setup as in Sect. \ref{sec:FC of Eisenstein series}. In particular, we exclude the odd orthogonal case. Then
\[
\dim\Hom_{N_{\scw}^{\heartsuit}}(I(s,\theta),\psi_\scw^{\heartsuit})=1.
\]
\end{Prop}

\begin{proof}
It is equivalent to show that the twisted Jacquet module $J_{N_{\scw}^{\heartsuit},\psi_{\scw}^{\heartsuit}}(I(s,\theta))$ is one-dimensional. As the number of orbit $P\bs G^{\square,k}/N_{\scw}^{\heartsuit}$ is infinite, one cannot apply the geometric lemma in \cite{BZ77} directly. One solution is to use the approach in \cite{Karel79}. The following proof is based the localization principle of Bernstein.

We write $P=M\cdot N=M(W^{\Delta,k})\cdot N(W^{\Delta,k})$. Given $g\in G^{\square,k}$, let $i_g=\ind_{N_{\scw}^{\heartsuit} \cap g^{-1}Pg}^{N_{\scw}^{\heartsuit}}(g^{-1}(\theta))$. If $\psi_{\scw}^{\heartsuit}|_{N_{\scw}^{\heartsuit}\cap g^{-1}Ng}$ is nontrivial, then $\Hom_{N_{\scw}^{\heartsuit}}(i_g,\psi_{\scw}^{\heartsuit})=0$. If $\psi_{\scw}^{\heartsuit}|_{N_{\scw}^{\heartsuit}\cap g^{-1}Ng}$ is trivial but induces a pair that is higher than the orbit $(k^n)_D$, then again $\Hom_{N_{\scw}^{\heartsuit}}(i_g,\psi_{\scw}^{\heartsuit})=0$. Thus, from the argument in Sect. \ref{sec:unfolding of FC}, we know that if $g\notin P\cdot P_{\scw}^{\heartsuit}$, then $\Hom_{N_{\scw}^{\heartsuit}}(i_g,\psi_{\scw}^{\heartsuit})=0$.

Let $d=\dim P\bs G^{\square,k}$. The induced representation $I(s,\theta)$ admits a filtration of $P_{\scw}^{\heartsuit}$-stable subspaces
\[
I^{(d)}(s,\theta)\subset \cdots \subset I^{(0)}(s,\theta)=I(s,\theta),
\]
where
\[
I^{(m)}(s,\theta)=\left\{f^{(s)}\in I(s,\theta)\mid \mathrm{Supp}(f^{(s)})\subset \bigcup_{\dim P\bs PwP_{\scw}^{\heartsuit}\geq m} PwP_{\scw}^{\heartsuit}\right\}.
\]
The quotient $I^{(m)}/I^{(m+1)}$ is the direct sum of spaces $I_w$, as $w$ ranges over a set of representatives $w$ of elements of $W(M)\bs W(G^{\square,k})$ with $\dim P\bs PwP_{\scw}^{\heartsuit}=m$, and where
\[
I_w=\ind_{w^{-1}Pw\cap P_{\scw}^{\heartsuit}}^{P_{\scw}^{\heartsuit}}(w^{-1}(\theta)).
\]

We now apply results from the previous section.
Let
\[
X=w^{-1}Pw\cap P_{\scw}^{\heartsuit}\bs P_{\scw}^{\heartsuit} \text{ and } Y=w^{-1}Pw\cap P_{\scw}^{\heartsuit}\bs P_{\scw}^{\heartsuit} /N_{\scw}^{\heartsuit}.
\] Then it is easy to check that $Y$ is homeomorphic to $w^{-1}Pw\cap M_{\scw}^{\heartsuit}\bs M_{\scw}^{\heartsuit}$ and is therefore an $l$-space. Let $p:X\to Y$ be the projection map.

We fix $w$ such that $\dim P\bs PwP_{\scw}^{\heartsuit}<d$. Observe that this is equivalent to $w\notin P\cdot P_{\scw}^{\heartsuit}$. We show that $J_{N_{\scw}^{\heartsuit},\psi_{\scw}^{\heartsuit}}(I_w)=0$, or equivalently $\Hom_{N_{\scw}^{\heartsuit}}(I_w,\psi_{\scw}^{\heartsuit})=0$. Recall that $I_w$ is the space of sections with compact support of an $l$-sheaf $\scf^w$ on $X=w^{-1}Pw\cap P_{\scw}^{\heartsuit}\bs P_{\scw}^{\heartsuit}$. An element in $\Hom_{N_{\scw}^{\heartsuit}}(I_w,\psi_{\scw}^{\heartsuit})$ is the same as a distribution $\scd$ on $\scf^w$ such that
\begin{equation}\label{eq:equi condition distribution}
u\cdot \scd = \psi_{\scw}^{\heartsuit}(u)\scd \text{ for all } u\in N_{\scw}^{\heartsuit}.
\end{equation}

We now use the localization principle. Let $y\in Y$. The restriction functor gives an $l$-sheaf in $\Sh(p^{-1}(y),N_{\scw}^{\heartsuit})$. Let $\tilde y\in P_{\scw}^{\heartsuit}$ be a representative of $y$. In our case,
\[
p^{-1}(y)=w^{-1}Pw\cap N_{\scw}^{\heartsuit}\bs \tilde y N_{\scw}^{\heartsuit}\simeq (w\tilde y)^{-1}P(w\tilde y)\cap N_{\scw}^{\heartsuit}\bs N_{\scw}^{\heartsuit}.
\]
By Proposition \ref{prop:composition of functors}, the representation obtained is the induced representation $i_{w\tilde y}$. By our discussion above, there are no nonzero distributions in $\fd(p^{-1}(y),\scf_{p^{-1}(y)})$ that satisfy \eqref{eq:equi condition distribution}. By Theorem \ref{thm:localization principle}, there is no nonzero distribution on $\scf^w$ such that \eqref{eq:equi condition distribution} holds.

We have shown that $J_{N_{\scw}^{\heartsuit},\psi_{\scw}^{\heartsuit}}(I_w)=0$ when $\dim P\bs PwP_{\scw}^{\heartsuit}<d$.
By exactness of the twisted Jacquet functor, the inclusion $I^{(d)}(s,\theta)\hookrightarrow I(s,\theta)$ induces an isomorphism
\[
J_{N_{\scw}^{\heartsuit},\psi_{\scw}^{\heartsuit}}(I^{(d)}(s,\theta))\cong J_{N_{\scw}^{\heartsuit},\psi_{\scw}^{\heartsuit}}(I(s,\theta)).
\]
The left-hand side is isomorphic to $J_{N_{\scw}^{\heartsuit}\cap P,\psi_{\scw}^{\heartsuit}}(\theta)$, which is one-dimensional since $\theta$ is a representation of type $(k,n)_D$.
\end{proof}

We now state a local result whose proof is close to the unfolding argument, at least for $\pi$ is a supercuspidal representation. We first define the character
$\kappa_\theta=\kappa_\theta^{\scw}$ of $G\times G$ by
\[
\kappa_\theta(g_1,g_2)=\chi_\theta(\nu(g_2)).
\]

\begin{Prop}
If $\pi$ is supercuspidal, then
\begin{equation}\label{eq:local multiplicity one}
\dim \Hom_{G\times G}(J_{N_{\scw}^{\bullet},\psi_{\scw}^{\bullet}}(I(s,\theta)),\kappa_\theta(\pi\boxtimes \pi^\vee))\leq 1.
\end{equation}
\end{Prop}

\begin{proof}
We first argue as in the previous result. The induced representation $I(s,\theta)$ admits a filtration of $P_{\scw}^{\bullet}$-stable subspaces
\[
I^{(d)}(s,\theta)\subset \cdots \subset I^{(0)}(s,\theta)=I(s,\theta),
\]
where
\[
I^{(m)}(s,\theta)=\left\{f^{(s)}\in I(s,\theta)\mid \mathrm{Supp}(f^{(s)})\subset \bigcup_{\dim P\bs PwP_{\scw}^{\bullet}\geq m} PwP_{\scw}^{\bullet}\right\}.
\]
(Note that the meaning of these notations are different from those in the proof of Proposition \ref{prop:unfolding FC local}.) Using the argument Proposition \ref{prop:unfolding FC local}, together with appropriate changes as in Sect. \ref{sec:unfolding}, we can show that $I^{(d)}(s,\theta)\hookrightarrow I(s,\theta)$ induces an isomorphism
\[
J_{N_{\scw}^{\bullet},\psi_{\scw}^{\bullet}}(I^{(d)}(s,\theta))\simeq J_{N_{\scw}^{\bullet},\psi_{\scw}^{\bullet}}(I(s,\theta)).
\]
Here, $I^{(d)}(s,\theta)\simeq \ind_{P\cap P_{\scw}^{\bullet}}^{P_{\scw}^{\bullet}}(\theta)$.

Next, we follow the proof in \cite{HKS96} Sect. 4. We first want to understand the structure of $J_{N_{\scw}^{\bullet},\psi_{\scw}^{\bullet}}(I^{(d)}(s,\theta))$ as a representation of $G\times G$. Here, we can directly apply the Geometric Lemma \cite{BZ77} as
\[
P\cap P_{\scw}^{\bullet}\bs P_{\scw}^{\bullet}/\iota(G\times G)N_{\scw}^{\bullet} \simeq  P\bs P\cdot P_{\scw}^{\bullet}/\iota(G\times G)N_{\scw}^{\bullet}
\]
is a finite set. In Sect. \ref{sec:invariant} we have shown that an orbit in $P\bs P\cdot P_{\scw}^{\bullet}/\iota(G\times G)N_{\scw}^{\bullet}$ is determined by an invariant $\kappa(L)$.

Let $n_0$ be the Witt index of $\scw$. Then $0\leq \kappa(L)\leq n_0$. By \cite{BZ77} Theorem 5.2, $J_{N_{\scw}^{\bullet},\psi_{\scw}^{\bullet}}(I^{(d)}(s,\theta))$ admits a filtration by support
\[
J_{N_{\scw}^{\bullet},\psi_{\scw}^{\bullet}}(I^{(d)}(s,\theta))=J^{(n_0)}(s,\theta)\supset \cdots \supset J^{(0)}(s,\theta),
\]
such that
\[
Q^{(i)}(s,\theta):=J^{(i)}(s,\theta)/J^{(i-1)}(s,\theta)\simeq \ind_{\St_i}^{G\times G}(\theta_i).
\]
Here, $\St_i$ is the stabilizer of a representative $L_i$ such that $\kappa(L_i)=i$ and the representation $\theta_i$ is $J_{N_{\scw}^{\bullet}\cap P(L_i),\psi_{\scw}^{\bullet}}(\theta)$.

If $i>0$, we define $N^-$ as in Sect. \ref{sec:unfolding}. We can choose a representative $L_i$ so that $\iota(\{1\}\times N^-)$ lies in the stabilizer of the pair $(N_{\scw}^{\bullet}\cap P(L_i),\psi_{\scw}^{\bullet})$ and the action of $\iota(\{1\}\times N^-)$ on $J_{N_{\scw}^{\bullet}\cap P(L_i),\psi_{\scw}^{\bullet}}(\theta)$ is trivial. Since $\pi$ is supercuspidal, we know that $\Hom_{\{1\}\times N^-}(Q^{(i)}(s,\theta),\kappa_\theta(\pi\boxtimes \pi^\vee))=0$.

Now consider a nonzero homomorphism $\lam:J_{N_{\scw}^{\bullet},\psi_{\scw}^{\bullet}}(I(s,\theta))\to \kappa_\theta(\pi\boxtimes \pi^\vee)$. If the restriction of $\lam$ to $J^{(0)}(s,\theta)$ were zero, then there would a smallest $i$, with $i>0$, such that the restriction of $\lam$ to $J^{(i)}(s,\theta)$ is nonzero. This restriction would then factor through the quotient $Q^{(i)}(s,\theta)$, contradicting our assumption on $\pi$.
Thus under our assumption on $\pi$, the restriction map
\[
\Hom_{G\times G}(J_{N_{\scw}^{\bullet},\psi_{\scw}^{\bullet}}(I(s,\theta)),\kappa_\theta(\pi\boxtimes \pi^\vee))\to \Hom_{G\times G}(J^{(0)}(s,\theta),\kappa_\theta(\pi\boxtimes \pi^\vee))
\]
is injective.

For $i=0$, we can take $L_0=W^{\Delta,k}$ (or equivalently, $\gamma=1$) and thus $\St_0=G^\diamondsuit$. By our assumption, the representation $J_{N_{\scw}^{\bullet}\cap P(L_0),\psi_{\scw}^{\bullet}}(\theta)$ is simply $\chi_\theta$ and $J^{(0)}(s,\theta)=\ind_{G^\diamondsuit}^{G\times G}(\chi_\theta)$.

We finally have
\[
\Hom_{G\times G}(\ind_{G^\diamondsuit}^{G\times G}(\chi_\theta),\kappa_\theta(\pi\boxtimes \pi^\vee))=\Hom_{G\times G}(\kappa_\theta^{-1}(\pi^\vee \boxtimes \pi), \Ind_{G^\diamondsuit}^{G\times G}(\chi_\theta^{-1})).
\]
By the Frobenius reciprocity, this is isomorphic to $\Hom_{G^\diamondsuit}(\pi^\vee \boxtimes \pi, \bc)$, which is clearly one-dimensional. This completes the proof.
\end{proof}

\begin{Rem}
For general $\pi$, \eqref{eq:local multiplicity one} is true outside a set of discrete values of $s$. But we have to use some multiplicativity property of $\theta$, which is not discussed in this paper.
\end{Rem}

\bibliographystyle{alpha-abbrvsort}


\end{document}